\documentclass[11pt]{amsart}
\usepackage{amssymb,amsmath}
\usepackage[width=5.65in,height=8.0in,centering]{geometry}
\usepackage[colorlinks,plainpages,backref,urlcolor=blue]{hyperref}
\usepackage{subfigure}
\usepackage[all]{xy}
\usepackage{tikz}
\usetikzlibrary{matrix}

\newcommand{\C}{{\mathbb C}}
\renewcommand{\P}{{\mathbb P}}

\newcommand{\B}{{\mathcal B}}

\newcommand{\K}{{\Bbbk}}   % field
\newcommand{\KK}{{\mathbb K}}   % a bigger field
\renewcommand{\AA}{{\mathbb A}}   % affine space
\newcommand{\A}{{\mathcal A}}  % arrangement
\newcommand{\T}{{\overline{T}}} % complete principal truncation
\newcommand{\set}[1]{\left\{#1\right\}}
\newcommand{\abs}[1]{\left|#1\right|}
\newcommand{\p}{\overline{p}}  % projective quotient map
\newcommand{\CC}{{\mathcal C}}  % circuits

\newcommand{\cl}[1]{{\rm cl}(#1)}  % matroid closure

\newcommand{\mathscr}{\mathcal}

\DeclareMathOperator{\Tor}{Tor}

\DeclareMathOperator{\OT}{R}  % Orlik-Terao ring (R = std notation)
\DeclareMathOperator{\M}{\mathsf M}

\DeclareMathOperator{\codim}{codim}

\DeclareMathOperator{\Proj}{Proj} % forming a projective scheme
\DeclareMathOperator{\Spec}{Spec} % forming a affine scheme
\DeclareMathOperator{\bc}{{\mathbf{bc}}} % broken circuit complex
\DeclareMathOperator{\In}{In}    % initial in term order
\DeclareMathOperator{\lead}{Lt}  % leading term
\DeclareMathOperator{\Fitt}{Fitt}

\newtheorem{prop}{Proposition}[section]

\newtheorem{thm}[prop]{Theorem}
\newtheorem{lem}[prop]{Lemma}
\newtheorem{cor}[prop]{Corollary}

\theoremstyle{definition}
\newtheorem{defn}[prop]{Definition}
\newtheorem{rem}[prop]{Remark}
\newtheorem{exm}[prop]{Example}
\newtheorem{que}[prop]{Question}

\newenvironment{example}%
{\pushQED{\qed}\begin{exm}}%
{\popQED\end{exm}}  % end example environment with a qed-symbol
\newenvironment{remark}%
{\pushQED{\qed}\begin{rem}}%
{\popQED\end{rem}}  % end example environment with a qed-symbol
\newenvironment{question}%
{\pushQED{\qed}\begin{que}}%
{\popQED\end{que}}  % end example environment with a qed-symbol

\begin{document}

\title[Modular decomposition of the Orlik-Terao algebra]%
{Modular decomposition of the Orlik-Terao algebra of a hyperplane arrangement}

\author{Graham Denham}
\address{Department of Mathematics\\ The University of Western Ontario\\ London, Ontario N6A 5B7\\}
\email{gdenham@uwo.ca}

\author{Mehdi Garrousian}
\address{Department of Mathematics\\ The University of Western Ontario\\ London, Ontario N6A 5B7\\}
\email{mgarrou@uwo.ca}

\author{\c Stefan O. Toh\v aneanu}
\address{Department of Mathematics\\ The University of Western Ontario\\ London, Ontario N6A 5B7\\}
\email{stohanea@uwo.ca}


\begin{abstract}
Let $\A$ be a collection of $n$ linear hyperplanes in $\K^\ell$, where $\K$ is an algebraically closed field.  The Orlik-Terao algebra of $\A$ is the subalgebra $\OT(\A)$ of the rational functions generated by reciprocals of linear forms vanishing on hyperplanes of $\A$.  It determines an irreducible subvariety $Y(\A)$ of
$\P^{n-1}$.  We show that a flat $X$ of $\A$ is modular if and only if
$\OT(\A)$ is a split extension of the Orlik-Terao algebra of the subarrangement $\A_X$.  This provides another
refinement of Stanley's modular factorization theorem~\cite{stan71}
and a new characterization of modularity, similar in spirit to the fibration theorem of \cite{Par00}.  

We deduce that if $\A$ is supersolvable, then its Orlik-Terao algebra
is Koszul.  In certain cases, the algebra is also a complete intersection, and we characterize when this happens.  
\end{abstract}

\maketitle
\setcounter{tocdepth}{1}
\tableofcontents

%%%%%%%%%%%%%%%%%%%
%%%%%%%%%%%%%%%%%%%

\section{Introduction}
We begin with a brief description of the main constructions appearing in our paper.
\subsection{Algebras of reciprocals, and reciprocal planes}
\label{ss:intro1}
Let $\K$ be an algebraically closed field, and let $\B_n=\set{\hat{H}_1,\ldots,\hat{H}_n}$ denote the set of coordinate hyperplanes
in $\K^n$.  Let $V$ be a linear subspace of $\K^n$ of dimension $\ell$, with the property that $V\not\subseteq \hat{H}_i$ for any $i$.
Consider the set of hyperplanes $\A=\set{H_1,\ldots,H_n}$ in $V$, where $H_i=\hat{H_i}\cap V$.  The set $\A$ is a central, essential hyperplane arrangement of rank $\ell$, and any such arrangement arises this way.  Our default reference for facts about hyperplane arrangements will be the book \cite{OTbook}.

Let $f\colon V\hookrightarrow\K^n$ denote the inclusion.  Then
$f_i\in \K[V]$ is a linear map
for which $H_i=\ker f_i$, for $1\leq i\leq n$.  Our main object of study in this paper is the following.
\begin{defn}\label{def:OTalgebra}
The Orlik-Terao algebra of $\A$ is the subalgebra of $\K(V)$ generated
by reciprocals of the linear polynomials defining the hyperplanes:
\[
 \OT(\A):=\K[1/f_1,\ldots,1/f_n].
\]
\end{defn}
This algebra and certain Artinian quotients of it first appeared in work of Orlik and Terao~\cite{ot94}, in the context of hypergeometric functions.

Let $M(\A)=V-\bigcup_i H_i$ denote the complement of the hyperplane arrangement $\A$.  Then $M(\A)=V\cap (\K^*)^n$, where $(\K^*)^n$ is both an algebraic torus and the complement of the arrangement $\B_n$.
Let $i\colon \K^n\to\K^n$ denote the Cremona transformation given
by $i(y)=(y_1^{-1},\ldots,y_n^{-1})$.  Clearly the restriction of
$i$ to the torus $(\K^*)^n$ is a regular map, and $M(\A)\cong i(M(\A))$.
By construction, the Orlik-Terao algebra is the coordinate
ring of the closure
\[
Y(\A):=\overline{i(M(\A))}
\]
in $\K^n$.  We will
call the variety $Y(\A)$ the {\em reciprocal plane} of $\A$,
following the terminology of \cite{SSV11}, where it plays a role
in understanding the entropic discriminant.

Various authors have studied the Orlik-Terao algebra and reciprocal plane: see, for example, \cite{PS06,st,Loo03,HT03}.  The construction has received renewed attention recently in \cite{HK11,SSV11}.
A number of basic properties are known: for example, the Cohen-Macaulay property~\cite{PS06}, explicit equations for $Y(\A)$ (see \S\ref{ss:OTintro}), and the Hilbert series
of $\OT(\A)$, given in terms of matroid combinatorics:
\begin{equation}\label{eq:Terao_theorem}
h(\OT(\A),t)=\pi(\A,t/(1-t)),
\end{equation}
where $\pi(\A,t)$ is the Poincar\'e polynomial of the arrangement.
This was proven by Terao~\cite{tera02} for $\K$ of characteristic zero, and by Berget~\cite{berget} in general.
\subsection{Modular factorizations}
Let $L(\A)$ denote the intersection lattice of an arrangement $\A$.
A subspace $X\in L(\A)$ is said to be {\em modular} if
if $X+Y\in L(\A)$ for all subspaces $Y\in L(\A)$.
If $X\in L(\A)$ is modular, then Stanley's Factorization Theorem~\cite{stan71} states that the Poincar\'e polynomial $\pi(\A,t)$
is divisible by $\pi(\A_X,t)$.  In fact, a result of Brylawski~\cite{Bry75} accounts for the quotient: he shows that,
if $X$ is modular, then
\begin{equation}\label{eq:pi_factors}
\pi(\A,t)=\pi(\A_X,t)\pi(\T_X(\M(\A)),t)/(1+t),
\end{equation}
where $\T_X(\M(\A))$ denotes the complete principal truncation of the matroid $\M(\A)$ of $\A$ at $X$. (We refer to \cite{Oxbook} for definitions from matroid theory.)

Falk and Proudfoot~\cite{FP02} showed that, for complex arrangements, the factorization \eqref{eq:pi_factors} has a topological explanation.  For any $X\in L(\A)$, let $p_X\colon V\to V/X$ denote the
projection of $V\cong\C^{\ell}$ onto its quotient by the linear space $X$.  Then $p_X$ restricts to a map $p_X\mid_{M(\A)}\colon M(\A)\to M(\A_X)$.  If, moreover, $X$ is modular, the restriction is a locally
trivial fibre bundle (shown by Terao~\cite{Ter86} in the case where
$X$ is a coatom, and Paris~\cite{Par00} in general: a detailed proof
appears in \cite{FP02}.)

The fibres $\A_v$ are homeomorphic to the (projective) complement of a realization of the complete principal truncation $\T_X(\M(\A))$,
(by \cite[Th.~2.1]{FP02}). % careful with attributions
Since the Poincar\'e polynomial of a complex arrangement counts the
Betti numbers of its complement, Stanley and Brylawski's factorization
\eqref{eq:pi_factors} is equivalent to the fact that the
Serre spectral sequence of the fibration sequence
\begin{equation}\label{eq:modfib}
\xymatrix{
\P M(\A_v)\ar[r] & M(\A)\ar^-{p}[r] & M(\A_X)
}
\end{equation}
degenerates at $E_2$.

The projection $p$ also induces a map $Y(\A)\to Y(\A_X)$ of reciprocal
planes, and of coordinate rings $\OT(\A_X)\to \OT(\A)$.  One of our main results is that,
if $X$ is modular, then $\OT(\A)$ is a free $\OT(\A_X)$-module. More precisely, we give an isomorphism of $\OT(\A_X)$-modules,
\begin{equation}\label{eq:intro_split}
\OT(\A)\cong \OT(\A_X)\otimes_\K \OT_{[n]-[X]}(\A),
\end{equation}
where the algebra $\OT_{[n]-[X]}(\A)$ denotes the coordinate ring of the fibre over zero (Theorem~\ref{thm:modular_decomp}).  In order to interpret this algebra, we introduce a {\em relative} Orlik-Terao algebra $\OT_H(\A)$, for any $H\in\A$ (Definition~\ref{def:OT_0}).  The fibre arrangement $\A_v$ comes with a distinguished hyperplane $X$, and we show (Theorem~\ref{thm:fibre}) that
\[
\OT_{[n]-[X]}(\A)\cong \OT_{X}(\A_v).
\]
Then Stanley's formula~\eqref{eq:pi_factors} appears by comparing Hilbert series under the isomorphism \eqref{eq:intro_split}, using a relative version of Terao's formula~\eqref{eq:Terao_theorem}.

The isomorphism \eqref{eq:intro_split} has a combinatorial explanation.  Proudfoot and Speyer~\cite{PS06} showed that the Orlik-Terao algebra of an arrangement $\A$ is a flat deformation of the Stanley-Reisner ring of the broken circuit complex $\bc(\A)$ of $\A$.  Brylawski and Oxley~\cite[Th.~1.6]{BO81} found $X$ is a modular flat of $\A$ if and only if the broken circuit complex decomposes as a join of induced subcomplexes:
\begin{equation}\label{eq:bc_splits}
\bc(\A)\cong \bc(\A)\mid_{[X]} * \bc(\A)\mid_{[n]-[X]}.
\end{equation}

We show (Theorem~\ref{thm:bc_modular}) that a flat $X$ is modular
if and only if $\bc(\A)\mid_{[n]-[X]}\cong\bc_0(\T_X(\M(\A)))$, where
$\bc_0$ denotes the reduced broken circuit complex (\S\ref{ss:OTintro}): this is a self-contained combinatorial result that seems to have been anticipated in the introduction to \cite{Bry77}, but the statement appears to be new.
Then, since a join of simplicial complexes gives rise to a tensor product of Stanley-Reisner rings, the decomposition \eqref{eq:intro_split} can be obtained by a deformation argument, using \cite{PS06}.

In Section~\ref{sec:apps}, we extract some consequences of our modular decomposition.  An arrangement $\A$ is {\em supersolvable} if there exists a maximal modular chain in $L(\A)$: that is, modular subspaces $X_i\in L(\A)$, for $1\leq i\leq \ell$, for which $X_1<\ldots<X_\ell$.  We show (Theorem~\ref{thm:koszul}) that, if $\A$ is supersolvable, then $\OT(\A)$ is a Koszul algebra.  We also consider a further generalization to hypersolvable arrangements (introduced by Jambu
and Papadima~\cite{JP98}), and compute Poincar\'e-Betti series for
$\OT(\A)$ explicitly in the special case of generic arrangements (\S\ref{ss:generic}).

Again, we note a parallel with the topology of the complement.  The cohomology ring of the complement $M(\A)$ (the {\em Orlik-Solomon algebra}) has a presentation which is similar to that of $\OT(\A)$.
By way of comparison, this algebra is a deformation of the {\em exterior} Stanley-Reisner ring of $\bc(\A)$.  Shelton and Yuzvinsky~\cite{ShYuz} showed the cohomology ring is Koszul if $\A$ is supersolvable, so our Theorem~\ref{thm:koszul} may be regarded as a commutative analogue.

Complete intersections generated by quadrics are known to be Koszul.
Typically, however, $\OT(\A)$ is not a complete intersection, even
if $\A$ is supersolvable.  In Section~\ref{sec:ci}, we characterize those arrangements for which $\OT(\A)$ is a quadratic complete
intersection.  Using results of Falk~\cite{Fa01} as well as Sanyal, Sturmfels, and Vinzant~\cite{SSV11}, we find that such arrangements
are exactly the supersolvable arrangements with exponents at most $2$.

%%%%%%%%%%%%%%%%%%%
%%%%%%%%%%%%%%%%%%%

\section{Background}
\label{sec:proj}
\subsection{Projection to closed subarrangements}
\label{ss:mod_fib}
Here, we recall in some more detail the results of \cite{FP02,Ter86,Par00}.
As in \S\ref{ss:intro1}, let $\A$ denote a central, essential arrangement of $n$ hyperplanes in an $\ell$-dimensional linear subspace $V$ of $\K^n$.  A hyperplane arrangement can be regarded
as a linear realization of a matroid without loops, and we will
denote the underlying matroid of $\A$ by $\M(\A)$.
The diagonal action of $\K^*$ on $\K^n$ restricts to $M(\A)\subseteq V$: we will let $\P\A$ denote the corresponding set of projective hyperplanes in $\P V$, and let $\P M(\A)$ denote their complement.

We will order the hyperplanes of $\A$ and take the underlying set of
$\M(\A)$ to be the set $[n]:=\set{1,\ldots,n}$, regarded as integers
indexing ordered hyperplanes.  We will abuse notation slightly and
regard $L(\A)$ both as the intersection lattice of $\A$ and the lattice of flats of $\M(\A)$: when the distinction is required, 
we write $[X]:=\set{i\in [n]\colon H_i\leq X}$ for $X\in L(\A)$.
The {\em rank} of a flat is its codimension
in $V$: let $L_p(\A)$ denote the flats of rank $p$.  {\em Coatoms} are flats of rank $\ell-1$.
Let $\A_X$ denote the subarrangement of $\A$
indexed by $[X]$, regarded as a hyperplane arrangement in the linear
space $V/X$.  Its intersection lattice is the lower interval $[V,X]$ of $L(\A)$.  Let $p_X\colon\K^n\to \K^{[X]}$ denote the coordinate projection given by deleting coordinates for hyperplanes $H\not\leq X$.

For a point $y\in V$, we note $p_X(y)=0$ if and only if $y_i=0$ for all hyperplanes $H_i\leq X$.  Thus we may identify $p_X(V)$ with $V/X$, and restrict $p_X$ further to hyperplane complements,
$p_X\colon M(\A)\to M(\A_X)$.
The map $p_X$ is compatible with the $\K^*$ action, so it induces a
map $\p_X\colon \P M(\A)\to\P M(\A_X)$.
Let $Q=\prod_{i=1}^n f_i$.  For convenience, order the hyperplanes
so that $H_i\leq X$ if and only if $1\leq i\leq n_X$, where $n_X=\abs{\A_X}$, and let
$Q_X=\prod_{i=1}^{n_X} f_i$.  Then, as schemes, $\P M(\A)=\Proj (\K[V]_Q)$, and $\P M(\A_X)=\Proj(\K[V/X]_{Q_X})$.

We note that, for any (reduced) point $v\in \P M(\A_X)$, the fibre over $v$ is the complement of a hyperplane arrangement:
consider the point projectively as a map $v\colon\P_\KK^0\to \P M(\A_X)$, for some extension $\KK$ of $\K$,
given by a graded homomorphism $v^*\colon \K[V/X]_{Q_X}\to\KK[t]$.
Then the homogeneous coordinate ring of $p^{-1}(v)$ is
\[
\K[V]_Q\otimes_{\K[V/X]_{Q_X}} \KK[t]\cong
\KK[\AA^1\times X]_{Q_v},
\]
where the polynomial $Q_v$ may be chosen to be the reduced image of $Q$ in the tensor product.

We make the following definition, noting that our notation differs slightly from that of \cite{FP02}, in that our arrangements are always
central.
\begin{defn}\label{def:fibre arrangement}
For each point $v\in \P M(\A_X)$, let $\A_{v,X}$ denote the  arrangement in $\AA^1\times X$ defined by $Q_v$.  We will simply
write $\A_v$ when the choice of $X$ is understood.
By construction, $\P M(\A_{v,X})=p_X^{-1}(v)$.  We take the underlying set of its matroid to be $\set{0}\cup([n]-[X])$.
\end{defn}
\begin{example}\label{ex:a3vsx3}
Consider the rank-$3$ arrangements $A_3$ and $X_3$ given, respectively, by defining equations $Q_1=xyz(x-y)(x-z)(y-z)$ and
$Q_2=xyz(x+y)(x+z)(y+z)$.  In each case, let $X$ be the linear subspace given by $x=y=0$, so that $\A_X$ is the rank-$2$ arrangement of $3$ lines, for $\A=A_3$ and $\A=X_3$.

If $v=[\alpha\colon\beta]\in\P M(\A_X)$ is a closed point and $\A=A_3$, then $Q_v=ctz(z-\alpha t)(z-\beta t)$, where $c$ is a unit, and the arrangement $\P\A_v$ consists of four points in $\P_\K^1$.
On the other hand, if $\A=X_3$, then $Q_v=ctz(z+\alpha t)(z+\beta t)$, and $\P\A_v$ consists of four points in $\P_\K^1$ as long as
$\alpha\neq \beta$.  Figure~\ref{fig:proj} shows typical fibres for
both arrangements, where each complement $\P M\subseteq\P^2$ is drawn in the affine chart with $z=1$.

If $\eta$ is the generic point in $\P M(\A_X)$, let $\KK=\K(V/X)_{0}$ and map $\K[V/X]\to\KK[t]$ by $x\mapsto t$, $y\mapsto (y/x)t$.  Then
for $\A=A_3$, $Q_\eta$ is a unit multiple of $tz(z-t)(z-(y/x)t)$, so
$\P\A_\eta$ consists of four points in $\P_\KK^1$, and similarly for
$\A=X_3$.
\end{example}

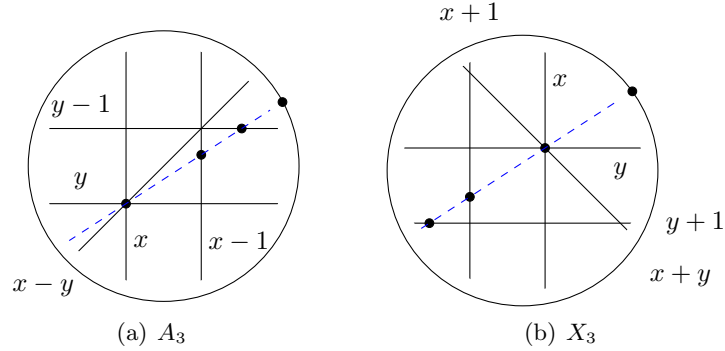
\begin{figure}
\centering
\subfigure[$A_3$]{
\begin{tikzpicture}%[scale=0.85]
\tikzstyle{every node}=[font=\small]
\draw (0.5,0.5) circle (1.8);
% labels
\node at (-0.6,0.3) {$y$};
\node at (-0.6,1.3) {$y-1$};
\node at (0.2,-0.5) {$x$};
\node at (1.5,-0.5) {$x-1$};
\node at (-1.1,-1.1) {$x-y$};
% intersections
\draw[fill=black] (0,0) circle (0.15em);
\draw[fill=black] (1,0.65) circle (0.15em);
\draw[fill=black] (1.538,1) circle (0.15em);
\draw[fill=black] (2.08,1.352) circle (0.15em);
% lines
\clip (0.5,0.5) circle (1.6);
\draw (0,-2.5) -- (0,2.5);
\draw (1,-2.5) -- (1,2.5);
\draw (-2.5,0) -- (2.5,0);
\draw (-2.5,1) -- (2.5,1);
\draw (-2.5,-2.5) -- (2.5,2.5);
\draw[style=dashed,color=blue] (-2,-1.3) -- (2,1.3);
\end{tikzpicture}
}
\qquad
\subfigure[$X_3$]{
\begin{tikzpicture}%[scale=0.85]
\tikzstyle{every node}=[font=\small]
\draw (-0.3,-0.3) circle (1.8);
% labels
\node at (1,-0.3) {$y$};
\node at (2,-1) {$y+1$};
\node at (0.2,0.9) {$x$};
\node at (-1,1.8) {$x+1$};
\node at (1.8,-1.7) {$x+y$};
% intersections
\draw[fill=black] (0,0) circle (0.15em);
\draw[fill=black] (-1,-0.65) circle (0.15em);
\draw[fill=black] (-1.538,-1) circle (0.15em);
\draw[fill=black] (1.16,0.754) circle (0.15em);
% lines
\clip (-0.3,-0.3) circle (1.6);
\draw (0,-2.5) -- (0,2.5);
\draw (-1,-2.5) -- (-1,2.5);
\draw (-2.5,0) -- (2.5,0);
\draw (-2.5,-1) -- (2.5,-1);
\draw (-2.5,2.5) -- (2.5,-2.5);
\draw[style=dashed,color=blue] (-2,-1.3) -- (2,1.3);
\end{tikzpicture}
}
\caption{Typical fibres for $x=y=0$}\label{fig:proj}
\end{figure}

In general, then, the underlying matroid of the arrangement $\A_v$ depends on the choice of $v$.
The typical value, however, is the {\em complete principal truncation} of $\M(\A)$ with respect to $X$.  We denote it by $\T_X(\M(\A))$: see, e.g.,~\cite[p.~379]{Oxbook}.

\begin{prop}\label{prop:fibre}
For any flat $X$ of an arrangement $\A$, there is a dense open subscheme $U$ of $\P M(\A_X)$ for which $\M(\A_{v,X})=\T_X(\M(\A))$,
for all points $v$ in $U$.
\end{prop}
\begin{proof}
As in \cite[Thm.~2.1]{FP02}, an arrangement $\A_{v,X}$ is a realization of the matroid $\T_X(\M(\A))$ as long as some finite
list of determinants depending on $v$ are all nonzero.
\end{proof}
\begin{defn}\label{def:TXA}
In particular, if $\eta\in\P M(\A_X)$ is the generic point, we always have $\M(\A_{\eta,X})=\T_X(\M(\A))$.  Let $\T_X(\A)=\A_{\eta,X}$.  This is an arrangement over the field $\K(V/X)_0$, and is in some sense the canonical realization of $\T_X(\M(\A))$, given $\A$.
\end{defn}

In the complex analytic case, one can say more, provided that $X$ is modular.

\begin{thm}[Theorem~2.4, \cite{FP02}]\label{thm:fibration}
If $\A$ is a complex arrangement and $X$ is modular,
the restriction of $p_X$ to $M(\A)$ is a fibre bundle projection.  The
fibres are homeomorphic to the complement of $\P\A_{v,X}$, for any choice of $v\in M(\A_X)$.
\end{thm}

In the (important) special case where $X$ is a modular coatom,
$\A_{v,X}$ is an arrangement of lines in $\C^2$, so the fibre
is the complement of finitely many points in $\P^1$.

\subsection{Broken circuits and the Orlik-Terao algebra}
\label{ss:OTintro}
For any arrangement $\A$, Proudfoot and Speyer~\cite{PS06} showed
that the Orlik-Terao algebra (Definition~\ref{def:OTalgebra}) admits the following presentation, for which the reader is also referred to \cite[Prop.~2.1]{st}.  Let $S:=\K[y_1,\ldots,y_n]$ denote the
coordinate ring of $\K^n$.  Then the inclusion $Y(\A)\subseteq \K^n$
induces a surjective homomorphism
\begin{equation}\label{eq:OTpres}
\K[y_1,\ldots,y_n]\to \OT(\A)
\end{equation}
sending $y_i$ to $1/f_i$, for $1\leq i\leq n$.  Let $I(\A)$ denote the
kernel of the map \eqref{eq:OTpres}.

For any $c\in\K^{n}$, denote its support by $[c]\subseteq[n]$.  The
{\em circuits} of $\A$ are those $c\in\K^{n}$ for which $\sum_{i=1}^n c_i f_i=0$ and $[c]$ is minimal.   If $c$ is a circuit, the element
\begin{equation}\label{eq:OTrels}
r_c:=\sum_{i\in [c]}c_i \prod_{j\in[c]-\set{i}} y_j
%y_1\cdots\widehat{y_i}\cdots y_k
\end{equation}
is easily seen to be in $I(\A)$.  $c$ is determined
up to a nonzero scalar multiple by its support: if $C$ is a circuit
of $\M(\A)$, let $r_C=r_c$ where $C=[c]$ and $r_c$ is monic.  Let
$\CC(\M(\A))$ denote the set of circuits of $\M(\A)$.

The relations $\set{r_c}$ are determinantal.
\begin{prop}\label{prop:fitting}
Let $X$ be a flat of $\A$ of rank $k$, and $n_X=\abs{X}$.  If $n_X>k\geq1$, let
\[
A_X=\left(\partial_j\log(f_i/f_{n_X})\right)_{1\leq i\leq n_X-1,\, 1\leq j\leq k}
\]
where $\set{x_1,\ldots,x_k}$ generate $\K[X]$, and $\partial_j:=\partial/\partial x_j$.  
Regarded as a matrix with (linear) entries in $S$,
\[
\Fitt_k(A_X)=(r_C\colon\text{$C\in\CC(\M(\A))$ and $\cl{C}=[X]$}),
\]
where $\cl{C}$ denotes the smallest flat containing $C$.
\end{prop}
\begin{proof}
By restricting to a closed subarrangement if necessary, we may assume $X$ has rank $k=\ell$ and $n_X=n$.  If $\A$ is not the Boolean arrangement, then $\ell<n$.  Let $J$ be the $n\times \ell$ matrix whose
$(i,j)$ entry is $\partial_j f_i$, and let ${\mathbf v}=(f_1,\ldots,f_n)^T$.  Since $f_i$ is linear for $1\leq i\leq n$, $J(x_1,\ldots,x_\ell)^T={\mathbf v}$.  Let $J'$ denote
the $n\times(\ell+1)$ matrix obtained from $J\mid {\mathbf v}$ by
dividing row $i$ by $f_i$.  Evaluated on $M(\A)$, the maximal minors
of $J'$ vanish.  In fact, if $C$ is a circuit of rank $\ell$,
$r_C$ is a unit multiple of the minor $J'_{C,[\ell+1]}$.  Now
subtract row $n$ of $J'$ from each row $i$ for $1\leq i\leq n-1$:
the result may be written
\[
\begin{tikzpicture}[scale=0.75,baseline=(current bounding box.center)]
\matrix (partials) [matrix of math nodes,
,left delimiter=(,right delimiter=)]
{
\partial_1\log(f_1/f_n) & \cdots &
\partial_{\ell}\log(f_1/f_n)  &[1em]  0\\
\vdots & \ddots & \vdots & \vdots \\
\partial_1\log(f_{n-1}/f_n) & \cdots &
\partial_\ell\log(f_{n-1}/f_n)  & 0 \\[5pt]
\partial_1\log f_n & \cdots & \partial_\ell\log f_n & 1\\
};
\draw[dashed] 
      ([xshift=-0.5em]partials-1-4.north west-|partials-4-4.south west) -|
      ([xshift=-0.5em]partials-4-4.south west);
\draw[dashed] 
      ([yshift=-0.4em]partials-3-1.south west-|partials-3-4.south east) -|
      ([yshift=-0.4em]partials-3-1.south west);
\end{tikzpicture}
\]

and the claim follows.
\end{proof}

These are, in fact, the only relations:
\begin{thm}[Theorem~4, \cite{PS06}]\label{thm:PS_grobner}
If $\A$ is an arrangement of $n$ hyperplanes, we have
$\OT(\A)\cong\K[y_1,\ldots,y_n]/I(\A)$, where
$I(\A)=(r_c\colon \text{$c$ is a circuit of $\A$})$.  Moreover,
the relations $\set{r_c}$ form a universal Gr\"obner basis for $\OT(\A)$ for which
\begin{equation}\label{eq:SR_def}
\In \OT(\A)\cong \K[y_1,\ldots,y_n]/J_{\bc(\A)},
\end{equation}
where $J_{\bc(\A)}$ is the Stanley-Reisner ideal of the broken circuit complex with respect to the given (arbitrary) ordering.
\end{thm}
To expand on the second statement, we recall the definition of the
broken circuit complex.  Fix any order of the hyperplanes $\A$.  A
subset of $[n]$ is a {\em broken circuit} if it is of the form
$C-\set{i}$, where $C$ is a circuit and $i$ is its least element.
A subset $I$ of $[n]$ is called a nbc-set if it does not contain
a broken circuit.  Clearly, nbc-sets are independent.  They
form an abstract simplicial complex on the vertex set $[n]$,
denoted $\bc(\A)$ and called the {\em broken circuit complex}.
Similarly, the {\em reduced broken circuit complex}, denoted $\bc_0(\A)$, consists of all subsets of $[n]-\set{e}$ not containing a broken circuit, where $e$ is the least element: see 
\cite[\S 7.4]{Bj94} and \cite{Bry77} for reference.  In particular, the complexes $\bc(\A)$ and $\bc_0(\A)$ are pure of dimension $\ell$ and $\ell-1$, respectively, where $\ell$ is the rank of $\A$.

If $\A$ is supersolvable, then the broken circuit complex decomposes
inductively as a join of zero-dimensional complexes, by \eqref{eq:bc_splits}, so in this case $\bc(\A)$ is a flag complex:
that is, minimal non-simplices have two vertices.  Consequently, the Stanley-Reisner ideal $J_{\bc(\A)}$ is generated in degree $2$, which is to say
that $\OT(\A)$ is a $G$-algebra: that is, it possesses a quadratic Gr\"obner basis.  This has the following consequence (see \cite{QAbook}):
\begin{thm}\label{thm:koszul}
If an arrangement $\A$ is supersolvable, then $\OT(\A)$ is a Koszul
algebra.
\end{thm}  
We will refine this result in Section~\ref{sec:apps}.

%%%%%%%%%%%%%%%%%%%
%%%%%%%%%%%%%%%%%%%

\section{An algebra factorization}
\label{sec:splits}
The goal of this section is to show that the Orlik-Terao algebra
$\OT(\A)$ is a split extension of $\OT(\A_X)$ if and only if $X$ is a modular flat of an arrangement $\A$, Theorem~\ref{thm:modular_decomp}.
The algebra decomposition comes from a reciprocal plane analogue of the Modular Fibration Theorem~\ref{thm:fibration}.  It depends on
the following combinatorial characterization of modularity: we show that the
(reduced) broken circuit complex of the complete principal truncation
is the subcomplex $\bc(\A)|_{[n]-[X]}$ exactly when $X$ is modular
(Theorem~\ref{thm:bc_modular}).

\subsection{Coordinate projections and intersections}
\label{ss:proj}
To resume the topic of \S\ref{ss:mod_fib}, suppose $I\subset[n]$
indexes a nonempty
subset of an arrangement $\A$, and let $p_I\colon\K^n\to\K^{I}$ be the induced coordinate projection.  Then $p_I$ restricts to a map of reciprocal
planes, $p_I|_{Y(\A)}\colon Y(\A)\to Y(\A_I)$.  The ring homomorphism $p^*_I\colon\OT(\A_I)\to\OT(\A)$ is obviously injective, so $p_I|_{Y(\A)}$ is dominant.

On the other hand, for $I\subseteq[n]$, let $\K^I\subseteq\K^n$ denote
the coordinate subspace supported on coordinates $I$, and $(\K^*)^I=\set{x\in\K^n\colon x_i\neq0\Leftrightarrow i\in I}$.
\begin{defn}\label{def:OT_0}
For any nonempty $I\subseteq[n]$, let $\OT_I(\A):=\OT(\A)/(y_i\colon i\in [n]-I)$.  Let $Y_I(\A)=\Spec \OT_I(\A)=Y(\A)\cap\K^{I}$,
the scheme-theoretic fibre over $0$ of the projection $p_{[n]-I}|_{Y(\A)}$.
If $I=[n]-\set{i}$ for some $i$, we will write $\OT_{H_i}(\A)$ in place of $\OT_{[n]-\set{i}}(\A)$, and call this the {\em relative Orlik-Terao algebra}.
\end{defn}

If $X$ is a flat of $\A$, Proudfoot and Speyer~\cite{PS06} showed that
$Y(\A_X)\cong Y_{[X]}(\A)$: in other words, the map $p^*_X$ is split by the homomorphism
\[
s_X(y_i)=
\begin{cases}
y_i&\text{if $H_i\leq X$;}\\
0 & \text{otherwise,}
\end{cases}
\]
for $1\leq i\leq n$.  Moreover, they showed $L(\A)$ indexes a stratification of $Y(\A)$, which leads to the following decomposition.
\begin{prop}
For any $I\subset[n]$,
\[
Y_I(\A)=\bigcup_{\substack{X\in L(\A):\\ I\subseteq[X]}} Y_{[X]}(\A)
\]
\end{prop}
\begin{proof}
If $I\subseteq J$, the identity map on $\K^n$ restricts to a
map $Y_I\to Y_J$, hence to
\[
\bigcup_{\substack{X\in L(\A):\\ I\subseteq[X]}} Y_{[X]}(\A)
\to
Y_I(\A).
\]
By \cite{PS06}, the open stratum $Y^\circ_I(\A):= Y(\A)\cap(\K^*)^{I}$ is empty, unless $I=[X]$ for some $X\in L(\A)$, so a localization argument shows the map above is in fact an isomorphism.
\end{proof}
Now we restrict our attention to the case in which $\abs{I}=n-1$.
\begin{cor}
For any $H\in\A$,
\[
Y_H(\A)\cong\bigcup_{\substack{X\in L(\A):\\ H\not\leq X}}Y(\A_X).
\]
\end{cor}
The two parts of Theorem~\ref{thm:PS_grobner} have counterparts.
\begin{prop}\label{prop:relative_relns}
For any hyperplane $H_i\in\A$, the kernel of the natural map
\[
\K[y_j\colon j\in[n],j\neq i ]\to \OT_{H_i}(\A)
\]
is generated by elements $\bar{r}_c$ indexed by circuits $c$ of $\A$:
\begin{equation}\label{eq:relOTrels}
\bar{r}_c:=\begin{cases}
\sum_{j\in [c]}c_j \prod_{k\in[c]-\set{j}} y_k &
\text{if $i\not\in [c]$;}\\
\prod_{j\in [c]-\set{i}} y_j & \text{if $i\in [c]$.}
\end{cases}
\end{equation}
\end{prop}

\begin{prop}\label{prop:bc_subcomplex}
Let $\A$ be an ordered arrangement of hyperplanes.  If we have
$I=\set{k,k+1,\ldots,n}$
for some integer $k$ with $1\leq k\leq n$, then with respect to 
lexicographic order,
\[
\In \OT_I(\A)\cong \K[y_1,\ldots,y_n]/J_{\bc(\A)|_{I}},
\]
where $J_{\bc(\A)|_{I}}$ denotes the Stanley-Reisner ideal of
the subcomplex of $\bc(\A)$ supported on vertices $I$.
\end{prop}

\begin{proof}
Since $\OT_I(\A)\cong \K[y_1,\ldots,y_n]/((y_1,\ldots,y_{k-1})+I(\A))$, by Theorem~\ref{thm:PS_grobner} it is enough to show that
\[
\In((y_1,\ldots,y_{k-1})+I(\A))=(y_1,\ldots,y_{k-1})+\In(I(\A)).
\]
The inclusion $\supseteq$ is immediate.  In the other direction, suppose $f\in (y_1,\ldots,y_{k-1})+I(\A)$.  We need to show that $\lead(f)\in (y_1,\ldots,y_{k-1})+\In(I(\A))$, where $\lead$ denotes the leading term in
lexicographic order.

Write $f=y_1g_1+\cdots+y_{k-1}g_{k-1}+h$ for some polynomials $g_i,h$, where $h\in I(\A)$.
If $h=0$, then $\lead(f)\in (y_1,\ldots,y_{k-1})$, and our claim is clear.  Otherwise, $f$ and $h$ have the same degree. Since the variables $y_1,\ldots,y_{k-1}$ are first in order, monomials of a fixed degree that are are in $(y_1,\ldots,y_{k-1})$ come before those that are not. It follows that $\lead(f)=\lead(h)$.  Since $h\in I(\A)$, the claim is shown.
\end{proof}

Since $\bc(\A)$ is a cone over $\bc_0(\A)$ with vertex $1$, by \cite{Bry77}, we obtain a relative version of Theorem~\ref{thm:PS_grobner}.
\begin{cor}\label{cor:relative_nbc}
If $H$ is first in order in an arrangement $\A$, then
\[
\In \OT_H(\A)\cong \K[y_1,\ldots,y_n]/J_{\bc_0(\A)},
\]
where $J_{\bc_0(\A)}$ is the Stanley-Reisner ideal of the reduced broken circuit complex.
\end{cor}

The Hilbert series formula \eqref{eq:Terao_theorem} also has a
relative version.  For any arrangement, the Poincar\'e polynomial
$\pi(\A,t)$ is divisible by $1+t$, and the polynomial
$\pi(\P\A,t):=\pi(\A,t)/(1+t)$ enumerates the Betti numbers of the
projective complement $\P M(\A)$ if $\A$ is a complex arrangement.
\begin{prop}\label{prop:relative_hs}
If $H$ is a hyperplane of an arrangement $\A$, then $\OT_H(\A)$ is
Cohen-Macaulay, and
\[
h(\OT_H(\A),t)=\pi(\P\A,t/(1-t)).
\]
\end{prop}
\begin{proof}
Since $\OT(\A)$ is a domain, the sequence
\begin{equation}\label{eqn:multbyy_i}
\xymatrix{
0\ar[r] & \OT(\A)[-1]\ar[r]^-{\cdot y_i} & \OT(\A)\ar[r] & \OT_H(\A)\ar[r] & 0
}
\end{equation}
is exact, where $H=H_i$.  The first result then follows from the Cohen-Macaulay property for
$\OT(\A)$ from \cite{PS06}, and the second from
\eqref{eq:Terao_theorem}.
\end{proof}
\begin{example}\label{ex:rel_rank2}
Suppose $\A$ is an arrangement of rank $2$.  Every $3$-element set $C\subseteq{2,\ldots,n}$ is a circuit, and
\[
\OT_{H_1}(\A)\cong\K[y_2,\ldots,y_n]/(y_iy_j\colon 2\leq i<j\leq n),
\]
using \eqref{eq:relOTrels} from Proposition~\ref{prop:relative_relns},
and
\[
h(\OT_{H_1}(\A),t)=1+(n-1)t/(1-t).\qedhere
\]
\end{example}

\subsection{Coordinate intersections and modular flats}
In this section, we compare the schemes $Y_{[n]-[X]}(\A)$ and $Y_X(\A_{v,X})$.  Recall $X$ is a hyperplane of the fibre arrangement
$\A_{v,X}$ (Definition~\ref{def:fibre arrangement}): the ambient affine space of both schemes, then, may be identified with $\K^{[n]-[X]}$.  Now we show that, for generic $v$, we have
$Y_{[n]-[X]}(\A)\cong Y_{X}(\A_{v,X})$ if and only if $X$ is a modular flat.  Our comparison has the following purely combinatorial foundation.  By reordering the hyperplanes of $\A$ if necessary, we will often assume that $[X]=\set{1,2,\ldots,n_X}$, a condition which we shall abbreviate by saying {\em $X$ is an initial flat of $\A$}.

\begin{thm}\label{thm:bc_modular}
Suppose $X$ is an initial flat of $\A$.  Then $\bc_0(\T_X(\A))$ is a
subcomplex of $\bc(\A)|_{[n]-[X]}$.  The two complexes are
equal if and only if $X$ is modular.
\end{thm}
Before giving the proof, we recall two facts about modular flats.  The
first is a description of the lattice $L(\T_X(\A))$ in the
modular case.
\begin{lem}[Prop.~5.14(3), \cite{Bry75}, Prop.~2.3, \cite{FP02}]
\label{lem:pr_trunc2}
If $X\in L(\A)$ is modular,
\[
L(\T_X(\A))\cong \set{Y\in L(\A)\colon X\wedge Y=V\text{~or~} X\leq Y}.
\]
Moreover, the rank function is given by
\[
\rho_{\T_X(\A)}(Y)=\begin{cases}\rho(Y) & \text{if $X\wedge Y=V$;}\\
\rho(Y)-\rho(X)+1 & \text{if $X\leq Y$.}
\end{cases}
\]
\end{lem}

Second, Brylawski's ``short-circuit axiom'' of \cite{Bry75} characterizes modularity in terms of circuits: $X$ is modular if and only if, for all $C\in \CC(\M(\A))$ for which $C-[X]\neq\emptyset$, there exists some $q\in [X]$ such that $(C-[X])\cup\set{q}$ is a
dependent set.  To reformulate slightly, let $\partial(S) := \cl{S}-S$, for any $S\subseteq[n]$.  Then the short-circuit axiom
can also be stated as
\begin{lem}\label{lem:shortcircuit}
An initial flat $X$ of $\M(\A)$ is modular if and only if, for all $C\in\CC(\M(\A))$, either $C\subseteq[X]$, $C\subseteq[n]-[X]$, or $\partial(C-[X])$ is nonempty and its least element is in $[X]$.
\end{lem}

It will be convenient to isolate a technical lemma:
\begin{lem}\label{lem:bc_modular}
Suppose that $X\in L(\A)$ is initial, and $\bc_0(\T_X(\A))=
\bc(\A)|_{[n]-[X]}$.  For any $C\in\CC(\M(\A))$ with $C\not\subseteq[X]$, the set $C-[X]$ is a broken circuit of $\M(\A)$.
\end{lem}
\begin{proof}
If $C\cap[X]=\emptyset$, the claim is obvious; otherwise, $(C-[X])\cup\set{0}$ is a circuit of $\T_X(\A)$.  Since $0<\min(C-[X])$, we see $C-[X]$ is a broken circuit of $\T_X(\A)$.  By
hypothesis, then, $C-[X]$ is also a broken circuit of $\A$.
\end{proof}
\begin{proof}[Proof of Theorem~\ref{thm:bc_modular}]
Fix $\A$ and $X$, and let $\Delta=\bc(\A)|_{[n]-[X]}$ and $\Delta_0=\bc_0(\T_X(\A))$ for
short.  First, we show that $\Delta_0\subseteq \Delta$ for all $X$.

If $\sigma\in\Delta_0$, suppose $\sigma\not\in\Delta$.  Since the set
$\sigma$ is independent in $\M(\A)$, by choosing $\sigma$ to be minimal, we may assume $\sigma$ is a broken circuit: that is, there exists some $i$ for which $\sigma\cup\set{i}\in\CC(\M(\A))$ and $i<\min(\sigma)$.  We must have $i\leq n_X$, since otherwise $\sigma\cup\set{i}$ would be a  circuit in $\T_X(\A)$, contrary to assumption.
But then $\sigma\cup\set{0}$ is dependent in $\T_X(\A)$, again a contradiction.  It follows that $\Delta_0\subseteq \Delta$.

Now suppose $X$ is modular.  We assume $\rho(X)>1$, since the case where $X$ is a hyperplane is trivial.  Suppose that $S\subseteq[n]-[X]$ is a minimal non-face of $\Delta_0$.  That is, $S$ is a broken circuit, meaning $S\cup\set{q}$ is a circuit of $\T_X(\A)$ for some $q<\min(S)$.  First, suppose $q=0$.
By Lemma~\ref{lem:pr_trunc2}, we must have $X\wedge \cl{S}\neq V$, so there
exists some $q\in [X]\cap\cl{S}$.  That is, $S\cup\set{q}$ is dependent in $\M(\A)$, so there is a circuit $C\subseteq S\cup\set{q}$.
Then $C-\set{q}\cup\set{0}$ is a dependent set in $\T_X(\A)$: by
assumption on $S$, we must have $C=S\cup\set{q}$.  Since $q\leq n_X<\min(S)$, we see $S$ is a broken circuit of $\M(\A)$, and a non-face of $\Delta$.

Otherwise, $q>n_X$.  If $S\cup\set{q}$ is a circuit of $\M(\A)$, again
we are done, so assume $S\cup\set{q}$ is independent in $\M(\A)$.
Let $Y_0=\cl{S}$ and $Y=\cl{S\cup\set{q}}$ in $L(\A)$.
Then $\rho(Y)=\abs{S}+1>\rho_{\T_X(\A)}(Y)$, so by Lemma~\ref{lem:pr_trunc2}, we must have $X\leq Y$.  Since $X$ is modular,
\begin{eqnarray*}
\rho(X\wedge Y_0) &=& \rho(X)+\rho(Y_0)-\rho(Y)\\
&=&\rho(X)-1\\
&\geq&1,
\end{eqnarray*}
by hypothesis.  It follows that  $S\cup\set{r}$ is dependent for some $r\in[X]$.  Since then $r<\min(S)$, we see $S$ is again a non-face of $\Delta$.

Finally, we check the converse.  Suppose that $\Delta_0=\Delta$, and we show $X$ is modular using Brylawski's ``short circuit axiom'' (Lemma~\ref{lem:shortcircuit}).
If $C\not\subseteq[X]$, then $\partial(C-[X])$ is nonempty by Lemma~\ref{lem:bc_modular}, so we must show $\min\partial(C-[X])\in[X]$.

Order subsets $S,T\subseteq[n]$ so that $S\prec T$ if
either $\abs{S}<\abs{T}$ or $\abs{S}=\abs{T}$ and $S$ precedes $T$ in
lexicographic order.  Let
\[
{\mathcal C}_X=\set{C\in\CC(\M(\A)):C\not\subseteq[X]\text{~and~}
\min\partial(C-[X])\not\in [X]}.
\]
If $X$ is not modular, then ${\mathcal C}_X$ is not empty, and it contains a minimal element $C_1$.
Let $q=\min\partial(C_1-[X])$: by assumption, $q>n_X$.
Then, for some $S\subseteq C_1-[X]$, we have a circuit $S\cup\set{q}$
with $q<\min S$.  Let $r=\max S$.  By the circuit exchange axiom, there exists a circuit
\begin{eqnarray*}
C_2&\subseteq& C_1\cup S\cup\set{q}-\set{r}\\
&=&C_1\cup\set{q}-\set{r}.
\end{eqnarray*}
By construction, $q<r$, so $C_2\prec C_1$.  If $C_2\subseteq[X]$, note that $C_2-\set{q}\subseteq C_1-[X]$, so circuit exchange using
$C_2$ and $S\cup\set{q}$ to eliminate $q$ would give a circuit contained in $C_1-[X]$, a contradiction.  So $C_2\not\subseteq[X]$.

Let $p=\min\partial(C_2-[X])$.  Then $T\cup\set{p}$ is a circuit of $\M(\A)$ for some broken circuit $T\subseteq C_2-[X]$.  By minimality of $C_1$ in ${\mathcal C}_X$, we have $p\in[X]$, so $p<q$.
Note that $q\in T$, since if not, $T\subseteq C_1-[X]$.  But then $p\in \partial(C_1-[X])$, contradicting our choice of $q$ (since $p<q$).

To complete the argument, we use the circuit exchange axiom again with $S\cup\set{q}$ and $T\cup\set{p}$ to obtain a circuit
\begin{eqnarray*}
C_3&\subseteq &S\cup\set{q}\cup T\cup\set{p}-\set{q}\\
&\subseteq &(C_1-[X])\cup\set{p}.
\end{eqnarray*}
Clearly $p\in C_3$, which contradicts the choice of $q$.  We conclude that ${\mathcal C}_X$ is empty.
\end{proof}
This leads to a comparison of coordinate intersections of reciprocal planes.
\begin{thm}\label{thm:fibre}
If $X$ is a flat of an arrangement $\A$, then for any
closed point $v$ in $\P M(\A_X)$, there is a surjection of algebras
\begin{equation}\label{eq:fibre_map}
p\colon \OT_{[n]-[X]}(\A)\to \OT_{X}(\A_v),
\end{equation}
The map $p$ is an isomorphism if and only if $X$ is a modular flat.
\end{thm}

\begin{proof}
Let $\phi\colon Y_{X}(\A_v)\to\K^{[n]-[X]}$ denote the natural embedding.  We claim that $\phi^*$ factors through $\OT_{[n]-[X]}(\A)$.
To see this, using \eqref{eq:OTpres} and Definition~\ref{def:OT_0}, we must check that the ideal $I(\A)+(y_i\colon i\in[X])$ maps to zero.  Since relations $r_c$ from
\eqref{eq:OTrels} indexed by circuits $c$ of $\A$ generate $I(\A)$, it is enough to show $\phi^*(r_c)=0$ for all $c$, and we do so by considering three cases.
\begin{enumerate}
\item $\abs{[c]\cap[X]}\geq2$: in this case, $r_c$ is zero in $\OT_{[n]-[X]}(\A)$, since each monomial contains a variable indexed by $[X]$.
\item $[c]\cap[X]=\set{i}$, for some $i$: then $\tilde{C}:=\set{0}\cup[c]-\set{i}$
is a circuit in $\T_X(\M(\A))$.  The image of the element $r_c$ in $\OT_{[X]}(\A)$ is a unit multiple of $\prod_{j\in [c]-\set{i}}y_i$ by \eqref{eq:OTrels}.  Then $\phi^*(r_{[c]})$ is zero, in view of the relation indexed by $\tilde{C}$ in \eqref{eq:relOTrels}.
\item $[c]\cap[X]=\emptyset$: then $c$ is also a circuit of $\A_v$, so the image of $r_c$ in $\OT_X(\A_v)$ is zero.
\end{enumerate}
Since $\phi^*$ is surjective, so is the induced map $p$
of \eqref{eq:fibre_map}.

To prove the second claim, we order the hyperplanes of $\A$ so that
$[X]$ is initial, and pass to initial ideals.  Using Proposition~\ref{prop:bc_subcomplex}, we see that $p$ induces
a map of Stanley-Reisner rings
\[
\In(p)\colon \K[y_1,\ldots,y_n]/J_{\bc(\A)_{[n]-[X]}}\to
\K[y_0,y_k,y_{k+1},\ldots,y_n]/J_{\bc_0(\A_{v,X})}.
\]
Since the map is the identity on the nonzero degree-$1$ elements, it induces a map of simplicial complexes.  If $X$ is modular, then,
$\In(p)$ is an isomorphism, by Theorem~\ref{thm:bc_modular},
then $p$ is as well.  Conversely, if $p$ is an isomorphism, then
$\bc(\A)_{[n]-[X]}\cong \bc_0(\A_{v})$.  It follows that
$\M(\A_v)=\T_X(\M(\A))$, using Proposition~\ref{prop:fibre}:
if not, $\M(\A_v)$ has strictly more dependent sets than $\T_X(\M(\A))$, so $\bc_0(\A_{v})$ would be a proper subcomplex
of $\bc_0(\T_X(\A))$, contradicting Theorem~\ref{thm:bc_modular}.

\end{proof}

It follows from Theorem~\ref{thm:fibre} that:
\begin{cor}
If $X$ is modular, the algebras $\OT_X(\A_v)$ are isomorphic, for all closed points $v\in \P M(\A_X)$.
\end{cor}

\subsection{An algebra decomposition}
Now we connect the results above.  Let $X$ be a flat of an arrangement $\A$, and let $\pi\colon \OT(\A)\to \OT_{[n]-[X]}(\A)$ be the natural surjection.  Let $v\in\P M(\A)$ be generic.  If $C\in\bc_0(\T_X(\A))$, then $C$ is also a simplex of $\bc(\M(\A))$, by Theorem~\ref{thm:bc_modular}.
By Corollary~\ref{cor:relative_nbc}, then, there is an
additive map $\iota\colon \OT_X(\A_v)\to \OT(\A)$ given by inclusion of monomials supported on broken circuits.  By construction, $\iota$ is a section of the surjection $p\circ\pi\colon \OT(\A)\to \OT_X(\A_v)$ from Theorem~\ref{thm:fibre}.
We define a map of $\OT(\A_X)$-modules,
\begin{equation}\label{eq:alg_map}
\tau_X\colon \OT(\A_X)\otimes_\K \OT_X(\A_v)\to \OT(\A)
\end{equation}
by setting $\tau_X(x\otimes z)= x\cdot \iota(z)$, for
$x\in\OT(\A_X)$ and $z\in\OT_X(\A_v)$.  This map turns out to be
most interesting when $X$ is modular.

\begin{thm}\label{thm:modular_decomp}
$X$ is a modular flat of $\A$ if and only if $\tau_X$ is an isomorphism of $\OT(\A_X)$-modules: $\OT(\A)\cong\OT(\A_X)\otimes_\K\OT_X(\A_v)$.
\end{thm}
\begin{proof}
Without loss, assume $X$ is initial.
Suppose $X$ is modular: then the broken circuit complex decomposes as a join, by \cite[Thm.~1.6]{BO81}.
\begin{eqnarray*}
\bc(\A)&\cong &\bc(\A_X)*\bc(\A)|_{[n]-[X]},\\
&\cong &\bc(\A)_{[X]}*\bc_0(\T_X(\A)),
\end{eqnarray*}
by Theorem~\ref{thm:bc_modular}.
Using Theorem~\ref{thm:PS_grobner} and isomorphism \eqref{eq:SR_def},
\begin{align}\nonumber
\In\OT(\A)\cong &\, \K[y_1,\ldots,y_n]/J_{\bc(\A)}\\ \nonumber
\cong &\, \K[y_1,\ldots,y_n]/(J_{\bc(\A_X)}+J_{\bc_0(\T_X(\A))}),\\
\cong &\,\In\OT(\A_X)\otimes_\K \In\OT_X(\A_v),
\label{eq:in_iso}
\end{align}
by \eqref{eq:SR_def} again and Corollary  \ref{cor:relative_nbc}.

If $y\in \OT(\A)$ is a monomial supported on a broken circuit, then,
we have $y=xz$ for such monomials $x\in\In\OT(\A_X)$ and $z\in\In\OT_X(\A_v)$.  By construction, $y=\tau_X(x\otimes z)$, so $\tau_X$ is surjective.  On the other hand, \eqref{eq:in_iso} shows the domain and codomain are additively isomorphic, so $\tau_X$ is an isomorphism of $\OT(\A_X)$-modules.

Conversely, if $\tau_X$ is an isomorphism, so is \eqref{eq:in_iso}, in which case $\bc(\A)$ decomposes as a join of $\bc(\A)_{[X]}$ with
another complex, in which case $X$ is modular by \cite[Thm.~1.6]{BO81}.
\end{proof}

\begin{cor}\label{cor:free}
If $X$ is a modular flat of $\A$, then $\OT(\A)$ is a free module over
$\OT(\A_X)$.
\end{cor}

Using the map \eqref{eq:fibre_map}, there is also a $\OT(\A_X)$-module map
\begin{equation}\label{eq:fibre_decomp}
\tau_X\circ(1\otimes p)\colon \OT(\A_X)\otimes_\K \OT_{[n]-[X]}(\A)\to \OT(\A).
\end{equation}

Combining Theorems~\ref{thm:fibre} and \ref{thm:modular_decomp}, we
obtain the following.
\begin{cor}\label{cor:fibre_decomp}
$X$ is a modular flat of $\A$ if and only if the map \eqref{eq:fibre_decomp} is an isomorphism.
\end{cor}
\begin{proof}
If $X$ is modular, then combining Theorems~\ref{thm:fibre} and \ref{thm:modular_decomp} shows \eqref{eq:fibre_decomp} is an isomorphism.  Conversely, suppose $\tau_X$ is an isomorphism.  Ordering
the hyperplanes of $\A$ so that the indices $[X]$ come first, we obtain an isomorphism
\[
\In(\OT(\A))\cong\In(\OT(\A_X))\otimes_\K \In(\OT_{[n]-[X]}(\A))
\]
By Proposition~\ref{prop:bc_subcomplex}, it follows that
the broken circuit complex decomposes as a join,
$\bc(\A)\cong\bc(\A)|_{[X]}*\bc(\A)|_{[n]-[X]}$.  By
\cite[Thm.~11]{BO81}, $X$ is modular.
\end{proof}

We note that the Orlik-Solomon algebra analogue of this result appears
as \cite[Cor.~5.5]{BZ91}.

\section{Applications}
\label{sec:apps}
The Modular Fibration Theorem~\ref{thm:fibration} was first formulated
in the case of $X$ a coatom, since then the fibre is
$1$-dimensional.  Similarly, our algebra decomposition
\eqref{eq:alg_map} can be refined in this case.
\subsection{The coatomic case}\label{ss:coatomic}
Suppose that $X$ is a modular coatom of an arrangement $\A$.
As usual, we order the hyperplanes of $\A$ so that
$\A_X=\set{H_1,\ldots,H_{n_X}}$.  For any $i,j$ satisfying $n_X<i<j\leq n$, since $X$ is modular, we have
\[
X+H_i\cap H_j=H,\quad\text{for some (unique) hyperplane $H\in\A_X$.}
\]
Let $i\circ j$ denote the integer in $[n_X]$ for which $H=H_{i\circ j}$.  Since $\set{i,j,i\circ j}$ is a circuit, we have
$f_{i\circ j}=a_{ij}f_i+b_{ij}f_j$, for some scalars $a_{ij}$ and $b_{ij}$, for all $n_X<i<j\leq n$.

\begin{thm}\label{thm:coatom_decomp}
Suppose $X$ is a modular coatom of $\A$.  There is an algebra isomorphism
\[
\OT(\A)\cong \OT(\A_X)[z_i\colon n_X<i\leq n]/J,
\]
where $J$ is the ideal generated by $z_iz_j-y_{i\circ j}(a_{ij} z_j+
b_{ij} z_i)$, for all $n_X<i<j\leq n$.
\end{thm}
\begin{proof}
By Theorem~\ref{thm:modular_decomp}, as $\OT(\A_X)$-modules,
\begin{eqnarray*}
\OT(\A)&\cong& \OT(\A_X)\otimes_\K \OT_X(\A_v)\\
&\cong&\OT(\A_X)\otimes_\K \K[z_i\colon n_X<i\leq n]/(z_iz_j\colon
n_X<i<j\leq n),
\end{eqnarray*}
using Example~\ref{ex:rel_rank2}.  To specify the structure of
$\OT(\A)$ as an $\OT(\A_X)$-algebra, it is enough to observe that
the relations $r_C$ from \eqref{eq:OTrels} for the circuits
$C=\set{i,j,i\circ j}$ generate the kernel of the natural
map $\OT(\A_X)[z_i\colon n_X<i\leq n]\to \OT(\A)$.
\end{proof}
\subsection{$\Tor$ algebras and the Koszul property}
\label{ss:Tor}
If $X$ is a flat of $\A$, we consider the Eilenberg-Moore spectral
sequence of the fibre sequence
\[
\OT(\A_X)\hookrightarrow \OT(\A)\twoheadrightarrow \OT_{[n]-[X]}(\A).
\]
If $X$ is modular, by Corollary~\ref{cor:free}, $\OT(\A)$ is a free
$\OT(\A_X)$-module.  So by \cite[XVI.6.1]{CEbook}, we have
\begin{equation}\label{eq:TorSS}
E^2_{pq}=\Tor^{\OT_{[n]-[X]}(\A)}_p(\Tor^{\OT(\A_X)}_q(A,\K),B)\Rightarrow
\Tor^{\OT(\A)}_{p+q}(A,B)
\end{equation}
for any $\OT(\A_X)$-module $A$ and $\OT(\A)$-module $B$.
Recall that a standard graded $\K$-algebra $A$ is {\em Koszul} if
$\Tor^A_p(\K,\K)_q=0$ for $p\neq q$.
\begin{lem}\label{lem:rel_koszul}
Let $H$ be a hyperplane of an arrangement $\A$.  Then $\OT(\A)$ is Koszul if and only if $\OT_H(\A)$ is Koszul.
\end{lem}
\begin{proof}
From the short exact sequence \eqref{eqn:multbyy_i}, the
map $\OT(\A)\to \OT_H(\A)$ is Koszul, and the claim follows by \cite[\S2.5, Example 1]{QAbook}.
\end{proof}

\begin{thm}\label{thm:2of3_Koszul}
Suppose that $X$ is a modular flat of $\A$.  If $\OT_X(\A_v)$ and
$\OT(\A_X)$ are Koszul, then so is $\OT(\A)$.
\end{thm}
\begin{proof}
By Theorem~\ref{thm:fibre}, we have $\OT_{[n]-[X]}(\A)\cong \OT_X(\A_v)$,
which is Koszul if and only if $\OT(\A_v)$ is, by Lemma~\ref{lem:rel_koszul}.  The claim follows directly by examining the grading in \eqref{eq:TorSS}, taking $A=C=\K$.
\end{proof}
For rank-$2$ arrangements, we may compute directly.
\begin{lem}\label{lem:rank2koszul}
If $\A$ is an arrangement of rank $\ell\leq2$, then $\OT(\A)$ is Koszul.
\end{lem}
\begin{proof}
If $\ell=1$, the claim is immediate.
If $\ell=2$, by Example~\ref{ex:rel_rank2}, for any $H\in\A$, the algebra $\OT_H(\A)$ is a quotient of a polynomial ring by quadratic monomials.
It follows $\OT_H(\A)$ is Koszul, by \cite{An86}, and so is $\OT(\A)$, by Lemma~\ref{lem:rel_koszul}.
\end{proof}

Note that combining the last two statements gives an inductive proof
of Theorem~\ref{thm:koszul}.

\begin{example}\label{ex:X2}
The $X_2$ arrangement is defined by $Q=xyz(x+y)(x-z)(y-z)(x+y-2z)$.  A computation using Macaulay~2~\cite{M2} shows that the Orlik-Terao algebra is quadratic, but not Koszul.
\end{example}

An arrangement is said to be {\em $2$-formal} if the vector space of linear relations amongst the polynomials $\set{f_i}_{i=1}^n$ are generated by those from circuits of size three: see \cite{st} for a full discussion of this property.

\begin{cor}\label{cor:implications}
For any arrangement $\A$ we have
\begin{gather*}
\text{$\A$ is supersolvable}\Rightarrow \text{$\OT(\A)$ is a $G$-algebra}\Rightarrow\\
\text{$\OT(\A)$ is Koszul}\Rightarrow \text{$\OT(\A)$ is quadratic} \Rightarrow \text{$\A$ is $2$-formal.}
\end{gather*}
The last three implications are strict.
\end{cor}
\begin{proof}
The two implications come from Theorem~\ref{thm:koszul}.  The next is immediate, and Example~\ref{ex:X2} shows it is not reversible.  For the last implication, let $I_2(\A)$ denote the ideal generated by the degree-$2$ elements of $I(\A)$.  Then
if $I(\A)$ is quadratic, $\codim(I_2(\A))=\codim(I(\A))=n-\ell$.  By \cite[Theorem~2.4]{st}, this is equivalent to $\A$ being 2-formal.

However, the converse is not true: consider the non-Fano arrangement $\A$ (\cite{st}, Example 1.7), defined by $Q(\A)=xyz(x-y)(x-z)(y-z)(x+y-z)$.  Although $\A$ is $2$-formal, $I(\A)$ is not quadratic.
\end{proof}
\begin{question}\label{ques:koszul}
Do there exist arrangements $\A$ for which $\OT(\A)$ is Koszul, yet $\A$ is not supersolvable?  The analogous problem is also open for cohomology rings (\cite[\S5]{ShYuz}.)
\end{question}

\subsection{Restrictions and resolutions}
\label{ss:restriction}
Even in some cases for which the Orlik-Terao algebra is not Koszul, it is still possible to describe part of the algebra $\Tor^{\OT(\A)}(\K,\K)$ explicitly.

Suppose that $W\subseteq V$ are both linear subspaces of $\K^n$
and $V\not\subseteq \hat{H}_i$ for $1\leq i\leq n$, as in \S\ref{ss:intro1}.  If $\A=\set{\hat{H_i}\cap V\colon i\in [n]}$ as before, let $I=\set{i\in[n]\colon W\not\subseteq\hat{H_i}}$.
arrangement $\A^W:=\set{\hat{H_i}\cap W\colon i\in I}$ is called the
{\em restriction of $\A$ to $W$}.  Let $M(\A^W)=W\cap(\K^*)^{I}$,
the complement of the hyperplanes $\A^W$.  We say that the restriction $\A^W$ is $k$-generic if $L_p(\A^W)\cong L_p(\A)$ for $p\leq k$.
We note that the condition that $\A^W$ is $1$-generic is equivalent to $I=[n]$, in which case $Y(\A^W)$ is a subscheme of $Y(\A)$.

\begin{prop}\label{prop:1generic}
If $\A^W$ is a $1$-generic restriction, then the map $\OT(\A)\to \OT(\A^W)$ is surjective.  The kernel is
\[
I_\A(\A^W):=\left(\set{r_c\colon \text{$c$ is a circuit of $\A^W$ and not of $\A$}}\right).
\]
If, moreover, $\A^W$ is a $p$-generic restriction for $p\geq1$, then
$I_\A(\A^W)$ is generated in degrees strictly greater than $p$.
\end{prop}
\begin{proof}
Follows immediately from Theorem~\ref{thm:PS_grobner}.
\end{proof}

If $\A^W$ is a $2$-generic restriction, then any circuits of $\A^W$ which are not circuits of $\A$ have at least $4$ elements.  It follows that the ideal $I_\A(\A^W)$ is zero in degrees $\leq 2$.
Since $\Tor_p^{\OT(\A)}(\K,\K)_p$, depends only
on the degree $\leq2$ part of $\OT(\A)$, for $p\geq0$, we have
\[
\Tor^{\OT(\A^W)}_p(\K,\K)_p\cong \Tor^{\OT(\A)}_p(\K,\K)_p,
\]
for all $p\geq 0$: equivalently, the quadratic dual algebras are isomorophic, $\OT(\A)^!\cong \OT(\A^W)^!$.

The family of hypersolvable arrangements, introduced by Jambu and Papadima in~\cite{JP98}, interpolates between supersolvable and generic arrangements.  Like the former, they are defined recursively;
however, we use an equivalent characterization from \cite{JP02}, summarized
in \cite[Thm.~4.2]{DS06}.  An arrangement $\A$ is {\em hypersolvable}
if $\A=\B^W$, where $\B$ is supersolvable, and $\B^W$ is a $2$-generic restriction.  If $\A$ is hypersolvable, then, the linear strand of the resolution of $\K$ over $\OT(\A)$ or, equivalently, $\Tor_p^{\OT(\A)}(\K,\K)_p$ for $p\geq0$, agrees with that of the supersolvable arrangement $\B$.  From Proposition~\ref{prop:1generic},
$\OT(\A)\cong\OT(\B)/I_\B(\A)$, where $I_\B(\A)$ is generated in degrees $3$ and higher.

Recall that an arrangement $\A$ is called {\em generic} if it is an $\ell-1$-generic restriction of a Boolean arrangement, where $\ell$ is the rank of $\A$.  Such an arrangement is hypersolvable, provided that $\ell\geq3$: in this case, the supersolvable counterpart is the Boolean arrangement,
$\OT(\B)=\K[y_1,\ldots,y_n]$, and $I_\B(\A)=I(\A)$ is generated in degree $\ell$.  Schenck showed
in \cite[Thm.~3.7]{Sch11} that the regularity of $\OT(\A)$ for any arrangement $\A$ is bounded above by $\ell-1$: although he states the result for $\K=\C$, the same argument clearly works in general.  In
this case, it follows that $I(\A)$ has a linear resolution, so
the natural map
\[
\C[y_1,\ldots,y_n]\twoheadrightarrow \OT(\A)
\]
is Golod~\cite{BF85}.  Using this, we are able to compute the Betti-Poincar\'e series of
$\Tor^{\OT(\A)}(\K,\K)$ for all generic arrangements, and we do so
in the next section.

For more general hypersolvable arrangements, we lack information about
the regularity of the ideal $I_\B(\A)$, which prevents us from carrying out a similar analysis to the one in \cite{DS06} for $\Tor$ algebras of Orlik-Solomon algebras of hypersolvable arrangements.  By analogy, it seems reasonable to expect that $\OT(\A)$ is $(\ell-1)$-regular over $\OT(\B)$.

\subsection{Generic arrangements}
\label{ss:generic}
By Proposition~\ref{prop:fitting}, $I(\A)$ contains the $\ell$th fitting ideal of a $(n-1)\times\ell$ matrix $A_0$ over $\K[y_1,\ldots,y_n]$,
provided that $n>\ell$.  For $n>\ell\geq3$, let $\A_{n,\ell}$ denote a
generic arrangement of $n$ hyperplanes of rank $\ell$.  The only circuits of $\A_{n,\ell}$ have $\ell+1$ elements, so in fact we have
an equality, $I(\A_{n,\ell})=\Fitt_{\ell}(A_0)$.
Since $I(\A_{n,\ell})$ has codimension $n-\ell$, it has an Eagon-Northcott resolution: for details, we refer to \cite[\S A2.6]{eisen}.  This was observed first in the case $\ell=3$ in \cite{st}.

Let $S=\OT(\B)=\K[y_1,\ldots,y_n]$, and let
\[
Q_{n,\ell}(t)=\sum_{p\geq0}\dim_\K\Tor_p^{S}(I(\A_{n,\ell}),\K)_{p+\ell}t^p.
\]
An Eagon-Northcott resolution gives rise to Betti numbers
\begin{equation}\label{eq:Pnl}
Q_{n,\ell}(t)=\sum_{p=0}^{n-\ell-1}{n-1\choose \ell+p}{\ell-1+p\choose\ell-1}t^p.
\end{equation}

Let
\[
P_{n,\ell}(s,t)=\sum_{p,q\geq0}\dim_\K
(\Tor^{\OT(\A_{n,\ell})}_p(\K,\K)_q) s^pt^q
\]
denote the Poincar\'e-Betti series of $\OT(\A_{n,\ell})$.
Then the Golod property implies the following.
\begin{cor}
If $\A_{n,\ell}$ is a generic arrangement and $n>\ell\geq3$, the Betti-Poincar\'e series of $\Tor^{\OT(\A)}(\K,\K)$, is given by
\[
P_{n,\ell}(s,t)=\frac{(1+st)^n}{1-s^2t^\ell Q_{n,\ell}(st)}.
\]
\end{cor}
While it is not clear that this formula admits any interesting simplification, we remark that the Betti numbers of the ideals $I(\A_{n,\ell})$ assemble together into a simple generating function.
\begin{prop}
For all $n\geq\ell\geq3$, $Q_{n,\ell}(t)$ is the coefficient of $x^\ell y^n$ in the formal power series
\[
\frac{y}{1-y}\cdot\frac{1-(1+t)y}{1-(1+t+x)y}.
\]
\end{prop}
\begin{proof}
Use binomial expansions to simplify:
\begin{align*}
\sum_{n,\ell,p\geq0}{n-1\choose\ell+p}{\ell-1+p\choose\ell-1}x^\ell y^n t^p  = &
\sum_{\ell,p\geq0}
\frac{y^{\ell+1+p}}{(1-y)^{\ell+1+p}}{\ell-1+p\choose\ell-1} t^px^\ell\\
=&\frac{y}{1-y}\sum_{\ell\geq0}\Big(\frac{xy}{1-y}\Big)^\ell
\frac{1}{(1-ty/(1-y))^{\ell}}\\
=&\frac{y}{1-y}\cdot\frac{1-(1+t)y}{1-(1+t+x)y}.\qedhere
\end{align*}
\end{proof}
\section{Complete intersections}
\label{sec:ci}

It is known in general that a quadratic algebra that is a complete intersection is Koszul (see \cite{QAbook}, Section 2.6, Example 2).
Since the Orlik-Terao ideal $I(\A)$ has codimension $n-\ell$, the
algebra $\OT(\A)$ is Koszul if $I(\A)$ is generated by $n-\ell$
quadrics.  This is clearly a rather special situation, and we
characterize the arrangements for which this happens.  
We show (Theorem~\ref{thm:qCI2})
that $I(\A)$ is a quadratic complete intersection (``q.c.i.'') if and only if $\A$ is supersolvable with exponents $1$ and $2$.  (So, if $\OT(\A)$ is a Koszul complete intersection, then $\A$ is supersolvable: see Question~\ref{ques:koszul}.)

\subsection{A numerical constraint}
Recall that for $X\in L_2(\A)$, the number of hyperplanes containing $X$ is $1-\mu_\A(X)$.  If $\mu_\A(X)=-2$, we will call $X$ a {\em triple point}.  

\begin{thm}\label{thm:qCI}
Let $\A$ be an arrangement of $n$ hyperplanes, of rank $\ell$ such that $\OT(\A)$ is quadratic. Then the following are equivalent:
\begin{enumerate}
\item $\OT(\A)$ is a complete intersection.
\item For all $X\in L_2(\A)$, we have $\abs{\mu_\A(X)}\leq2$, and the number of triple points equals $n-\ell$.
%For every $X\in L_2(\A)$ we have $\mu_{\A}(X)\leq 2$ and $|\{X\in %L_2(\A): \mu_{\A}(X)=2\}|=n-\ell$.
\item $\pi(\A,t)=(1+t)^{2\ell-n}(1+2t)^{n-\ell}$. In particular $\ell\leq n\leq2\ell-1$.
\end{enumerate}
\end{thm}
\begin{proof} $(2)\Rightarrow (1)$. The proof is immediate from \cite[Proposition 2.1]{st}.
$(1)\Rightarrow (2)$. We argue by contradiction: suppose $X\in L_2(\A)$ has  $[X]=\set{1,\ldots,p}$, for some $p>3$.  Let $I(\A_X)\subset\K[y_1,\ldots,y_{p}]$ be the Orlik-Terao ideal of $\A_X$.  Since $p \geq 4$, we have four circuits on $\set{1,2,3,4}$
and four elements of $I(\A_X)$, which we abbreviate by $r_{123}$,
$r_{124}$, $r_{134}$, and $r_{234}$.
It is not difficult to see that any one of these is a linear combination of the other three, and there exists a linear syzygy on any three, so $I(\A_X)$ is (minimally) generated by $r_{1ij}$ for
$1<i<j\leq p+1$. 

By hypothesis, $I(\A)$ is generated by relations $r_C$, where $\abs{C}=3$.  Since such circuits intersect on at most
one element, monomials $y_1y_i$ for $1<i\leq p$ appear only in relations indexed by circuits of $\A_X$.  It follows that $r_{1ij}$
must be part of any minimal generating set of $I(\A)$ as well.
Since the quadrics $r_{123}$, $r_{124}$, $r_{134}$ have a linear syzygy, they cannot be part of a regular sequence, so if $I(\A)$ is a 
complete intersection, necessarily $p\leq 3$.

To complete the argument, note that, since rank-$2$ flats contain at most $3$ points, the relations
$\set{r_C\colon C\in\CC(\M(\A)), \abs{C}=3}$ are linearly
independent.  Then, since $I(\A)$ is a quadratic complete intersection, the number of triple points equals $\codim(I(\A))$, which is $n-\ell$.

$(1)\Rightarrow (3)$: Suppose that $\OT(\A)$ is a q.c.i.  Then  $\codim(I(\A))=n-\ell$, so $\OT(\A)$ has the Koszul graded minimal free resolution as a $S=\K[y_1,\ldots,y_n]$-module
\[0\rightarrow S(-2(n-\ell))\rightarrow\cdots\rightarrow S^{n-\ell}(-2)\rightarrow S\rightarrow \OT(\A)\rightarrow 0.\]
So the Hilbert series is
\begin{equation}\label{eq:hsofci}
h(\OT(\A),t)=\frac{(1-t^2)^{n-\ell}}{(1-t)^n}=\frac{(1+t)^{n-\ell}}{(1-t)^{\ell}}.
\end{equation}
By Terao's formula \eqref{eq:Terao_theorem}, this is equivalent to
\[
\pi(\A,t)=(1+t)^{2\ell-n}(1+2t)^{n-\ell}.
\]
Since $\A$ is a central arrangement, $1+t$ always divides this polynomial, and therefore $2\ell-n\geq 1$ and $n-\ell\geq 0$.

%\vskip .1in

$(3)\Rightarrow (1)$: Suppose we have a rank-$\ell$, central essential arrangement with Poincar\'e polynomial as in $(3)$.  Then $h(\OT(\A),t)$ is given by \eqref{eq:hsofci}, so
\[
h(\OT(\A),t)=1+nt+(n(n+1)/2-(n-\ell))t^2+\cdots,
\]
and $I(\A)$ contains $n-\ell$ (independent) quadrics.  By assumption, these generate $I(\A)$, so it is a complete intersection, as in
\cite[Cor.~1.8]{st}.
\end{proof}
\subsection{$3$-tree arrangements}
Here, we characterize the supersolvable arrangements having exponents
which are at most $2$.
We will follow the terminology for hypergraphs of, e.g., \cite{JMM06}.  We briefly recall some definitions.
\begin{defn}\label{def:hypergraph}
A hypergraph $G=(V,E)$ is a pair of a set of vertices $V$ and edges
$E\subseteq 2^V$.  We assume that each edge $e\in E$ contains at least
two vertices.  A {\em walk} in $G$ is an alternating sequence of vertices and edges with the property that consecutive vertices are contained in the intermediate edge.  $G$ is {\em connected} if every pair of vertices is joined by a walk.  A {\em cycle} is a walk with
at least two edges that begins and ends at the same vertex and consists of, otherwise,
distinct edges and vertices.  A hypergraph $G$ with no cycles is called a {\em hyperforest}.  A connected hyperforest is a {\em hypertree}.  The {\em edge graph} of a hypergraph is the graph with vertices $E$ and edges $\set{e,e'}$ whenever $e\cap e'\neq\emptyset$.

We shall be exclusively interested in hypergraphs
for which $\abs{e}=3$ for each $e\in E$, which we will call $3$-graphs, $3$-trees, and $3$-forests, respectively.
\end{defn}
\begin{defn}\label{def:GofA}
Let $\A$ be an arrangement of $n$ hyperplanes.
Let $V=[n]$, and $E(\A)=E=\set{[X]\colon X\in L_2(\A)}$.
Let $G(\A):=(V,E)$, the hypergraph on codimension-$2$ flats of $\A$.
\end{defn}

Note that, if $G(\A)$ is a $3$-graph, then for any $e,e'\in E(\A)$, 
if $e\neq e'$, we must have $\abs{e\cap e'}\leq 1$.
If $G(\A)$ is connected and $\abs{E}=1$, clearly $\M(\A)=U_{2,3}$, the
uniform matroid of $3$ points in rank $2$.
More generally, a condition introduced by Falk \cite{Fa01} ensures 
$G(\A)$ determines the matroid of $\A$:
\begin{defn}\label{def:lineclosed}
If $\M$ is a matroid on $[n]$, a set $S\subseteq [n]$ is said to be
{\em line-closed} if $\cl{\set{i,j}}\subseteq S$ for every $i,j\in S$.
The matroid $\M$ is {\em line-closed} if all line-closed sets are
flats of $\M$.
\end{defn}
\begin{prop}
If $\M(\A)$ is line-closed and $G(\A)$ is a $3$-graph, then
a set $S\subseteq[n]$ is a flat of $\M(\A)$ if and only if $\abs{S\cap e}\neq2$ for any $e\in E(\A)$.
\end{prop}
\begin{proof}
Suppose $\abs{S\cap e}\neq2$ for any $e\in E(\A)$.  Since $G(\A)$ is a $3$-graph, this means $S$ is line-closed, hence a flat.
\end{proof}
\begin{thm}\label{thm:combinatorics}
For an irreducible arrangement $\A$, the following are equivalent:
\begin{enumerate}
\item $\M(\A)$ is line-closed and $G(\A)$ is a $3$-tree with $\ell-1$ edges.
\item $\A$ is supersolvable, with exponents $\set{1\cdot1, (\ell-1)\cdot 2}$.
\item $\M(\A)$ is an iterated parallel connection of $\ell-1$ uniform matroids $U_{2,3}$.
\end{enumerate}
\end{thm}
\begin{proof}
We prove ``$(1)\Rightarrow(2)$'' by induction on $\abs{E}$.  
The claim being trivial for $\abs{E}=0$, suppose it
holds for $3$-trees with $\abs{E}<m$, where $m>0$.  Suppose $\A$ has
$n$ hyperplanes and $G(\A)$ is a $3$-tree with $m$ edges.  The edge graph of $G(\A)$ is a tree: let
$\set{e,e'}$ be an edge for which $e$ has degree $1$, and
let $\set{i,j}=e-e'$.  Let $S=[n]-\set{i,j}$.  Clearly $S$ is a 
maximal proper line-closed subset, so $S=[X]$ for a coatom $X\in L(\A)$.  

The induced subgraph of $G(\A)$ on $S$ may
be identified with $G(\A_X)$.  Again, $G(\A_X)$ is a $3$-tree, so
$\A_X$ is supersolvable by induction.  
Now $e=\set{i,j,k}$ for some $k$, a circuit of $\M(\A)$.
Then $X\wedge(\set{i}\vee\set{ j})=\set{k}$, so $X$ is modular.   It follows that $\A$ is 
supersolvable, and the exponent $2$ occurs with multiplicity one more
than in $\A_X$.

To show ``$(2)\Rightarrow(3)$'' use induction on $\ell$.  The base
case being obvious, suppose $\A_X$ is supersolvable with exponents
$\set{1\cdot 1, (\ell-2)\cdot2}$ and $X$ is modular.  Let $\set{i,j}=[n]-[X]$, and let $e=\cl{i,j}$.  Since $X$ is modular, 
$e=\set{i,j,k}$ for some $k\neq i,j$.  The submatroid on $e$ is
isomorphic to $U_{2,3}$, and the contraction $\M(\A)/k$ is disconnected, so $\M(\A)$ is a parallel connection of $\M(\A_X)$ with $U_{2,3}$ over $k$, by \cite[7.1.16]{Oxbook}.

We omit the implication ``$(3)\Rightarrow(1)$'', which is routine.
\end{proof}
Further discussion of the iterated parallel connection operad may be found in \cite{DS12}.  Property $(3)$ implies that $\M(\A)$ is graphic: the graph may be constructed by iterated parallel connection of triangles along the vertices of $G(\A)$.
\begin{figure}
\[
G(\A):\quad
\begin{tikzpicture}[scale=1.1,baseline=(current bounding box.center)]
%% make a new node shape for the triangular edges in a 3-graph
\makeatletter
\pgfdeclareshape{edge}{
\inheritsavedanchors[from=circle]
\inheritanchor[from=circle]{center}
\anchor{left}{
\pgf@x=-6pt%
\pgf@y=-3.46pt
}
\anchor{right}{
\pgf@x=6pt%
\pgf@y=-3.46pt%
}
\anchor{top}{
\pgf@x=0pt%
\pgf@y=6.93pt
}
\backgroundpath{
\pgfplothandlerclosedcurve
\pgfplotstreamstart
\pgfplotstreampoint{\pgfpoint{-10pt}{-5.77pt}}
\pgfplotstreampoint{\pgfpoint{10pt}{-5.77pt}}
\pgfplotstreampoint{\pgfpoint{0pt}{11.55pt}}
\pgfplotstreamend
\pgfusepath{stroke}
\pgfpathcircle{\pgfpoint{-6pt}{-3.46pt}}{1.5pt}
\pgfpathcircle{\pgfpoint{6pt}{-3.46pt}}{1.5pt}
\pgfpathcircle{\pgfpoint{0pt}{6.93pt}}{1.5pt}
\pgfusepath{fill}
}
}
\makeatother
%% now draw a representative 3-tree
%%
\pgfnode{edge}{left}{}{x}{\pgfusepath{stroke}}
{
\pgftransformrotate{-45}
\pgfnode{edge}{top}{}{y}{\pgfusepath{stroke}}
}
{
\pgftransformshift{\pgfpointanchor{x}{right}}
\pgftransformrotate{45}
\pgfnode{edge}{top}{}{z}{\pgfusepath{stroke}}
}
{
\pgftransformshift{\pgfpointanchor{y}{left}}
\pgftransformrotate{-90}
\pgfnode{edge}{top}{}{u}{\pgfusepath{stroke}}
}
{
\pgftransformshift{\pgfpointanchor{y}{right}}
\pgftransformrotate{10}
\pgfnode{edge}{top}{}{v}{\pgfusepath{stroke}}
}
{
\pgftransformshift{\pgfpointanchor{z}{right}}
\pgftransformrotate{40}
\pgfnode{edge}{top}{}{r}{\pgfusepath{stroke}}
}
{
\pgftransformshift{\pgfpointanchor{z}{right}}
\pgftransformrotate{155}
\pgfnode{edge}{top}{}{s}{\pgfusepath{stroke}}
}
\end{tikzpicture}
\qquad
\begin{tabular}{l}
$n=15$, $\ell=8$,\\
$\pi(\A,t)=(1+t)(1+2t)^7$.
\end{tabular}
\]
\caption{A $3$-tree arrangement}
\end{figure}
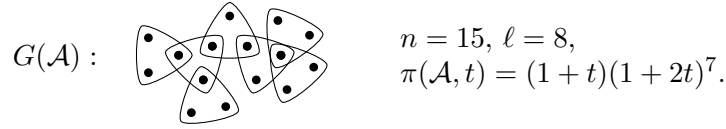
\subsection{A combinatorial characterization}
Now consider the hypergraph $G(\A)$ in the case that $\OT(\A)$ is a 
q.c.i.  By Theorem~\ref{thm:qCI}, we see $G(\A)$ is a $3$-graph with
$n-\ell$ edges.  The quadratic complete intersection property is
inherited by subarrangements:
\begin{lem}\label{lem:qci subarr}
If $\OT(\A)$ is a q.c.i., then so is $\OT(\A_X)$ for any $X\in L(\A)$.
\end{lem}
\begin{proof}
First, $I(\A_X)=I_2(\A_X)$: if not, suppose $r_C$ is 
an irredundant generator of degree $d\geq3$.  By reordering the hyperplanes, we assume the circuit $C=[d+1]$.
Since $\OT(\A)$ is quadratic, then $r_C=\sum_{k=1}^mP_kr_{C_k}$, for some circuits $C_k$ of size $3$, and some homogeneous polynomials $P_k\in \K[y_1,\ldots,y_n]$.

Choose a monomial order such that $\lead(r_C)=y_1\cdots y_d$.  Then there exist $1\leq i,j\leq d$, $i\neq j$, such that $i,j\in C_k$, for some $1\leq k\leq m$.  For this $k$, $\abs{C_k\cap C}\geq 2$. Since $\abs{C_k}=3$, this means that $C_k\subset C$, a contradiction.

So $I(\A_X)$ is a quadratic ideal, generated by some $\K$-linear combinations of the quadratic generators of $I(\A)$.  A subset of a 
regular sequence is a regular sequence, so $\OT(\A_X)$ is also a 
quadratic complete intersection.
\end{proof}
A reformulation of Proposition~21 of \cite{SSV11} gives:
\begin{prop}\label{prop:lineclosed}
If $\OT(\A)$ is quadratic, then $\M(\A)$ is line-closed.
\end{prop}
\begin{remark}
Falk showed in \cite{Fa01} that $\M(\A)$ is line-closed if the Orlik-Solomon algebra is quadratic, continuing our analogy.  Yuzvinsky
provided a counterexample to the converse: see \cite[Ex.~4.5]{DY01}.
It follows from \cite[Prop.~21]{SSV11} that, if $\M(\A)$ is line-closed, then $Y(\A)$ is cut out by degree-$2$ subideal $I_2(\A)$.  However, a Macaulay~2 \cite{M2} calculation using the same example from \cite{DY01} shows that the ideal generated by $I_2(\A)$ is not radical: as in the Orlik-Solomon case, then, the line-closed property does not imply quadraticity.
\end{remark}

\begin{lem}\label{lem:acyclic}
If $\OT(\A)$ is a q.c.i., then $G(\A)$ is a $3$-forest.
\end{lem}
\begin{proof}
Using Lemma~\ref{lem:qci subarr}, it is enough to show that if
$\M(\A)$ is irreducible and $\OT(\A)$ is a q.c.i., then $G(\A)$ is a
$3$-tree, since by Proposition~\ref{prop:lineclosed}, the connected 
components of $\M(\A)$ are those of the $3$-graph $G(\A)$.

Let $\abs{G(\A)}$ denote the simplicial complex on $V$ with
$2$-simplices $E$.  Since each $1$-simplex in
$\abs{G(\A)}$ is contained in exactly one $2$-simplex, $\abs{G(\A)}$ retracts onto its edge graph, a $1$-complex.  For the same reason, $\abs{G(\A)}$ is connected and simply connected if and only if $G(\A)$ is a hypertree.  Computing the Euler characteristic gives
\begin{eqnarray*}
b_1&=& 1-\abs{E}+3\abs{E}-n,\\
&=& n-2\ell+1,
\end{eqnarray*}
since $\abs{E(\A)}=n-\ell$, 
where $b_i$ denotes the $i$th Betti number of $\abs{G(\A)}$.
On the other hand, since $\M(\A)$ is irreducible, a classical result
due to Crapo implies that $\pi'(\A,1)\neq 0$; it follows from Theorem~\ref{thm:qCI}(3) that $2\ell-n=1$, so $b_1=0$, as required.
\end{proof}

We conclude with a combinatorial characterization.
\begin{thm}\label{thm:qCI2}
The following are equivalent.
\begin{enumerate}
\item $\OT(\A)$ is a q.c.i.
\item $\M(\A)$ is line-closed, and $G(\A)$ is a $3$-forest.
\item $\A$ is supersolvable, and its exponents are each $1$ or $2$.
\item Connected components of $\M(\A)$ are iterated parallel connections of uniform matroids $U_{23}$.
\end{enumerate}
\end{thm}
\begin{proof}
The implication ``$(1)\Rightarrow(2)$'' is given by Lemma~\ref{lem:acyclic} and Proposition~\ref{prop:lineclosed}. 
The last three conditions are equivalent, by Theorem~\ref{thm:combinatorics}.  If $\A$ is supersolvable, $\OT(\A)$ is quadratic (Theorem~\ref{thm:koszul}), so it is enough to note that if $\A$ has no exponent greater than $2$, then $\pi(\A,t)$ has the form of Theorem~\ref{thm:qCI}(3).
\end{proof}
In particular, we see that the quadratic complete intersection property depends only on the matroid of $\A$.

\begin{cor} 
Suppose $\A$ is an arrangement for which $\OT(\A)$ is a complete intersection.  Then the implications of Corollary \ref{cor:implications} are all equivalences.
\end{cor}
\begin{proof} 
If $\OT(\A)$ is a q.c.i., then it is supersolvable, so it remains only to show that if $\OT(\A)$ is a complete intersection and $\A$ is $2$-formal, then $\OT(\A)$ is quadratic.
Suppose $I(\A)$ is a complete intersection
with minimal generating set $\mathcal M$, and let $I_2(\A)$ be the subideal of $I(\A)$ generated by degree-$2$ elements.  By hypothesis, $\mathcal M$ forms a regular sequence.  Then $\mathcal M$ contains a minimal generating set $\mathcal N$ for $I_2(\A)$, since $I(\A)$ contains no elements of degree $<2$.

By \cite[Theorem~2.4]{st}, $2$-formality implies $\codim(I_2(\A))=\codim{I(\A)}=n-\ell$.   Since $\mathcal N$ must also
be a regular sequence, we have ${\mathcal N}={\mathcal M}$, and
$I(\A)$ is quadratic.
\end{proof}

%\bibliographystyle{amsalpha}
%\bibliography{aim}

\begin{thebibliography}{JMM06}

\bibitem[Ani86]{An86}
David~J. Anick, \emph{On the homology of associative algebras}, Trans. Amer.
  Math. Soc. \textbf{296} (1986), no.~2, 641--659. \MR{846601 (87i:16046)}

\bibitem[BF85]{BF85}
J{\"o}rgen Backelin and Ralf Fr{\"o}berg, \emph{Koszul algebras, {V}eronese
  subrings and rings with linear resolutions}, Rev. Roumaine Math. Pures Appl.
  \textbf{30} (1985), no.~2, 85--97. \MR{789425 (87c:16002)}

\bibitem[Ber10]{berget}
Andrew Berget, \emph{Products of linear forms and {T}utte polynomials},
  European J. Combin. \textbf{31} (2010), no.~7, 1924--1935. \MR{2673030
  (2011j:05062)}

\bibitem[Bj{\"o}92]{Bj94}
Anders Bj{\"o}rner, \emph{The homology and shellability of matroids and
  geometric lattices}, Matroid applications, Encyclopedia Math. Appl., vol.~40,
  Cambridge Univ. Press, Cambridge, 1992, pp.~226--283. \MR{1165544
  (94a:52030)}

\bibitem[Bry75]{Bry75}
Tom Brylawski, \emph{Modular constructions for combinatorial geometries},
  Trans. Amer. Math. Soc. \textbf{203} (1975), 1--44. \MR{0357163 (50 \#9631)}

\bibitem[Bry77]{Bry77}
\bysame, \emph{The broken-circuit complex}, Trans. Amer. Math. Soc.
  \textbf{234} (1977), no.~2, 417--433. \MR{468931 (80a:05055)}

\bibitem[BO81]{BO81}
Tom Brylawski and James Oxley, \emph{The broken-circuit complex: its structure
  and factorizations}, European J. Combin. \textbf{2} (1981), no.~2, 107--121.
  \MR{622075 (82i:05025)}

\bibitem[BZ91]{BZ91}
Anders Bj{\"o}rner and G{\"u}nter~M. Ziegler, \emph{Broken circuit complexes:
  factorizations and generalizations}, J. Combin. Theory Ser. B \textbf{51}
  (1991), no.~1, 96--126. \MR{1088629 (92b:52027)}

\bibitem[CE99]{CEbook}
Henri Cartan and Samuel Eilenberg, \emph{Homological algebra}, Princeton
  Landmarks in Mathematics, Princeton University Press, Princeton, NJ, 1999,
  With an appendix by David A. Buchsbaum, Reprint of the 1956 original.
  \MR{1731415 (2000h:18022)}

\bibitem[DS06]{DS06}
Graham Denham and Alexander~I. Suciu, \emph{On the homotopy {L}ie algebra of an
  arrangement}, Michigan Math. J. \textbf{54} (2006), no.~2, 319--340.
  \MR{2252762 (2007f:17039)}

\bibitem[DS12]{DS12}
Graham~{Denham} and Alexander~I. {Suciu}, \emph{{Multinets, parallel connections, and
  Milnor fibrations of arrangements}}, \href{http://www.arxiv.org/abs/1209.3414}{arXiv:1209.3414}, 2012.

\bibitem[DY02]{DY01}
Graham Denham and Sergey Yuzvinsky, \emph{Annihilators of {O}rlik-{S}olomon
  relations}, Adv. in Appl. Math. \textbf{28} (2002), no.~2, 231--249.
  \MR{1888846 (2003b:05046)}

\bibitem[Eis95]{eisen}
David Eisenbud, \emph{Commutative algebra}, Graduate Texts in Mathematics, vol.
  150, Springer-Verlag, New York, 1995, With a view toward algebraic geometry.
  \MR{1322960 (97a:13001)}

\bibitem[Fal02]{Fa01}
Michael Falk, \emph{Line-closed matroids, quadratic algebras, and formal
  arrangements}, Adv. in Appl. Math. \textbf{28} (2002), no.~2, 250--271.
  \MR{1888847 (2003a:05040)}

\bibitem[FP02]{FP02}
Michael~J. Falk and Nicholas~J. Proudfoot, \emph{Parallel connections and
  bundles of arrangements}, Topology Appl. \textbf{118} (2002), no.~1-2,
  65--83, Arrangements in Boston: a Conference on Hyperplane Arrangements
  (1999). \MR{1877716 (2002k:52033)}

\bibitem[GS]{M2}
Daniel~Grayson and Michael~Stillman, \emph{Macaulay2---a software system for algebraic
  geometry and commutative algebra}, available at \href{http://www.math.uiuc.edu/Macaulay2}{http://www.math.uiuc.edu/Macaulay2}.

\bibitem[HT03]{HT03}
Hiroki Horiuchi and Hiroaki Terao, \emph{The {P}oincar\'e series of the algebra
  of rational functions which are regular outside hyperplanes}, J. Algebra
  \textbf{266} (2003), no.~1, 169--179. \MR{1994536 (2004k:13031)}

\bibitem[HK11]{HK11}
June~{Huh} and Eric~{Katz}, \emph{{Log-concavity of characteristic polynomials and
  the Bergman fan of matroids}}, \href{http://www.arxiv.org/abs/1104.2519}{arXiv:1104.2519}, 2011.

\bibitem[JP98]{JP98}
Michel Jambu and Stefan Papadima, \emph{A generalization of fiber-type
  arrangements and a new deformation method}, Topology \textbf{37} (1998),
  no.~6, 1135--1164. \MR{1632975 (99g:52019)}

\bibitem[JP02]{JP02}
Michel Jambu and {\c{S}}tefan Papadima, \emph{Deformations of hypersolvable
  arrangements}, Topology Appl. \textbf{118} (2002), no.~1-2, 103--111,
  Arrangements in Boston: a Conference on Hyperplane Arrangements (1999).
  \MR{1877718 (2003a:32047)}

\bibitem[JMM06]{JMM06}
Craig Jensen, Jon McCammond, and John Meier, \emph{The integral cohomology of
  the group of loops}, Geom. Topol. \textbf{10} (2006), 759--784. \MR{2240905
  (2007c:20121)}

\bibitem[Loo03]{Loo03}
Eduard Looijenga, \emph{Compactifications defined by arrangements. {I}. {T}he
  ball quotient case}, Duke Math. J. \textbf{118} (2003), no.~1, 151--187.
  \MR{1978885 (2004i:14042a)}

\bibitem[OT92]{OTbook}
Peter Orlik and Hiroaki Terao, \emph{Arrangements of hyperplanes}, Grundlehren
  der Mathematischen Wissenschaften [Fundamental Principles of Mathematical
  Sciences], vol. 300, Springer-Verlag, Berlin, 1992. \MR{MR1217488
  (94e:52014)}

\bibitem[OT94]{ot94}
\bysame, \emph{Commutative algebras for arrangements}, Nagoya Math. J.
  \textbf{134} (1994), 65--73. \MR{1280653 (95j:52025)}

\bibitem[Oxl11]{Oxbook}
James Oxley, \emph{Matroid theory}, second ed., Oxford Graduate Texts in
  Mathematics, vol.~21, Oxford University Press, Oxford, 2011. \MR{2849819}

\bibitem[Par00]{Par00}
Luis Paris, \emph{Intersection subgroups of complex hyperplane arrangements},
  Topology Appl. \textbf{105} (2000), no.~3, 319--343. \MR{1769026
  (2002g:52031)}

\bibitem[PP05]{QAbook}
Alexander Polishchuk and Leonid Positselski, \emph{Quadratic algebras},
  University Lecture Series, vol.~37, American Mathematical Society,
  Providence, RI, 2005. \MR{2177131 (2006f:16043)}

\bibitem[PS06]{PS06}
Nicholas Proudfoot and David Speyer, \emph{A broken circuit ring}, Beitr\"age
  Algebra Geom. \textbf{47} (2006), no.~1, 161--166. \MR{2246531 (2007c:13029)}

\bibitem[SSV11]{SSV11}
Raman~{Sanyal}, Bernd~{Sturmfels}, and Cynthia~{Vinzant}, \emph{{The entropic
  discriminant}}, \href{http://www.arxiv.org/abs/1108.2925}{arXiv:1108.2925}, 2011.

\bibitem[Sch11]{Sch11}
Hal Schenck, \emph{Resonance varieties via blowups of {$\Bbb P^2$} and
  scrolls}, Internat. Math. Res. Notices (2011), no.~20, 4756--4778. \MR{2844937
  (2012k:14075)}

\bibitem[ST09]{st}
Hal Schenck and {\c{S}}tefan~O. Toh{\v{a}}neanu, \emph{The {O}rlik-{T}erao
  algebra and 2-formality}, Math. Res. Lett. \textbf{16} (2009), no.~1,
  171--182. \MR{2480571 (2010h:13023)}

\bibitem[SY97]{ShYuz}
Brad Shelton and Sergey Yuzvinsky, \emph{Koszul algebras from graphs and
  hyperplane arrangements}, J. London Math. Soc. (2) \textbf{56} (1997), no.~3,
  477--490. \MR{1610447 (99c:16044)}

\bibitem[Sta72]{stan71}
Richard~P. Stanley, \emph{Modular elements of geometric lattices}, Algebra
  Universalis \textbf{1} (1971/72), 214--217. \MR{0295976 (45 \#5037)}

\bibitem[Ter86]{Ter86}
Hiroaki Terao, \emph{Modular elements of lattices and topological fibration},
  Adv. in Math. \textbf{62} (1986), no.~2, 135--154. \MR{865835 (88b:32032)}

\bibitem[Ter02]{tera02}
\bysame, \emph{Algebras generated by reciprocals of linear forms}, J. Algebra
  \textbf{250} (2002), no.~2, 549--558. \MR{1899865 (2003c:16052)}

\end{thebibliography}
%\end{document}

\providecommand{\bysame}{\leavevmode\hbox to3em{\hrulefill}\thinspace}
\providecommand{\MR}{\relax\ifhmode\unskip\space\fi MR }
% \MRhref is called by the amsart/book/proc definition of \MR.
\providecommand{\MRhref}[2]{%
  \href{http://www.ams.org/mathscinet-getitem?mr=#1}{#2}
}
\providecommand{\href}[2]{#2}

\end{document}